\documentclass[11pt,a4paper]{amsart}
\usepackage{amssymb,amsmath,layout}

\theoremstyle{plain}
\newtheorem{thrm}{Theorem}[section]
\newtheorem{lemma}[thrm]{Lemma}

\newtheorem{cor}[thrm]{Corollary}
\newtheorem{rmrk}[thrm]{Remark}
\newtheorem{dfn}[thrm]{Definition}
\newtheorem{prob}[thrm]{Problem}

\begin{document}

\newcommand{\SL}{\mathcal L^{1,p}( D)}
\newcommand{\Lp}{L^p( Dega)}
\newcommand{\CO}{C^\infty_0( \Omega)}
\newcommand{\Rn}{\mathbb R^n}
\newcommand{\Rm}{\mathbb R^m}
\newcommand{\R}{\mathbb R}
\newcommand{\Om}{\Omega}
\newcommand{\Hn}{\mathbb H^n}
\newcommand{\aB}{\alpha B}
\newcommand{\eps}{\ve}
\newcommand{\BVX}{BV_X(\Omega)}
\newcommand{\p}{\partial}
\newcommand{\IO}{\int_\Omega}
\newcommand{\bG}{\boldsymbol{G}}
\newcommand{\bg}{\mathfrak g}
\newcommand{\bz}{\mathfrak z}
\newcommand{\bv}{\mathfrak v}
\newcommand{\Bux}{\mbox{Box}}
\newcommand{\e}{\ve}
\newcommand{\X}{\mathcal X}
\newcommand{\Y}{\mathcal Y}
\newcommand{\W}{\mathcal W}

\numberwithin{equation}{section}

\newcommand{\RN} {\mathbb{R}^N}
\newcommand{\Sob}{S^{1,p}(\Omega)}
\newcommand{\Dxk}{\frac{\partial}{\partial x_k}}
\newcommand{\Co}{C^\infty_0(\Omega)}
\newcommand{\Je}{J_\ve}
\newcommand{\beq}{\begin{equation}}
\newcommand{\bea}[1]{\begin{array}{#1} }
\newcommand{\eeq}{ \end{equation}}
\newcommand{\ea}{ \end{array}}
\newcommand{\eh}{\ve h}
\newcommand{\Dxi}{\frac{\partial}{\partial x_{i}}}
\newcommand{\Dyi}{\frac{\partial}{\partial y_{i}}}
\newcommand{\Dt}{\frac{\partial}{\partial t}}
\newcommand{\aBa}{(\alpha+1)B}
\newcommand{\GF}{\psi^{1+\frac{1}{2\alpha}}}
\newcommand{\GS}{\psi^{\frac12}}
\newcommand{\HFF}{\frac{\psi}{\rho}}
\newcommand{\HSS}{\frac{\psi}{\rho}}
\newcommand{\HFS}{\rho\psi^{\frac12-\frac{1}{2\alpha}}}
\newcommand{\HSF}{\frac{\psi^{\frac32+\frac{1}{2\alpha}}}{\rho}}
\newcommand{\AF}{\rho}
\newcommand{\AR}{\rho{\psi}^{\frac{1}{2}+\frac{1}{2\alpha}}}
\newcommand{\PF}{\alpha\frac{\psi}{|x|}}
\newcommand{\PS}{\alpha\frac{\psi}{\rho}}
\newcommand{\ds}{\displaystyle}
\newcommand{\Zt}{{\mathcal Z}^{t}}
\newcommand{\XPSI}{2\alpha\psi \begin{pmatrix} \frac{x}{|x|^2}\\ 0 \end{pmatrix} - 2\alpha\frac{{\psi}^2}{\rho^2}\begin{pmatrix} x \\ (\alpha +1)|x|^{-\alpha}y \end{pmatrix}}
\newcommand{\Z}{ \begin{pmatrix} x \\ (\alpha + 1)|x|^{-\alpha}y \end{pmatrix} }
\newcommand{\ZZ}{ \begin{pmatrix} xx^{t} & (\alpha + 1)|x|^{-\alpha}x y^{t}\\
     (\alpha + 1)|x|^{-\alpha}x^{t} y &   (\alpha + 1)^2  |x|^{-2\alpha}yy^{t}\end{pmatrix}}
\newcommand{\norm}[1]{\lVert#1 \rVert}
\newcommand{\ve}{\varepsilon}

\title[A parabolic analogue of the higher-order comparison theorem, etc.]{A parabolic analogue of the higher-order comparison theorem of De Silva and Savin}

\author{Agnid Banerjee}
\address{Department of Mathematics\\University of California, Irvine\\
CA 92697} \email[Agnid Banerjee]{agnidban@gmail.com}
\thanks{First author  supported in part by a post-doctoral grant of the Institute Mittag-Leffler's and by the second author's NSF Grant DMS-1001317}

\author{Nicola Garofalo}
\address{Dipartimento di Ingegneria Civile, Edile e Ambientale (DICEA) \\ Universit\`a di Padova\\ 35131 Padova, ITALY}
\email[Nicola Garofalo]{rembdrandt54@gmail.com}
\thanks{Second author supported in part by NSF Grant DMS-1001317 and by a grant of the University of Padova, ``Progetti d'Ateneo 2013"}

%
%
%

\maketitle

\begin{abstract}
We show that the quotient of two caloric functions which vanish  on a portion of the  lateral  boundary of  a $H^{k+ \alpha}$ domain is $H^{k+ \alpha}$ up to the boundary  for $k \geq 2$. In the case $k=1$, we show that the quotient is  in $H^{1+\alpha}$ if   the domain is assumed  to be space-time $C^{1, \alpha}$  regular. This can be thought of as a parabolic analogue of a recent important result in \cite{DS1}, and we closely follow the ideas in that paper. We also give  counterexamples  to the fact that  analogous results are not true at points on the parabolic  boundary which are not on the lateral boundary, i.e., points which are at the corner and base of the parabolic boundary.

\end{abstract}

\section{Introduction}
The classical comparison theorem states that two nonnegative harmonic functions which vanish on the boundary of a Lipschitz, or more in general a NTA domain, must vanish at the same rate. An important consequence of this result is that the quotient of two such functions is, in fact, H\"older continuous up to the boundary (only the function in the denominator needs now to be nonnegative now). In some recent remarkable works De Silva and Savin have established a higher-order version of such result. Specifically, they have proved in \cite{DS1} the following.

\begin{thrm}\label{t}
Let $D$ be a $C^{k, \alpha}$ domain in $\Rn$, with $0 \in \partial D$. Let $u, v$ be   two harmonic functions vanishing on $\partial D \cap B(0,1)$.  Furthermore,  $u>0$ in $D$ and $u=1$ at some interior point in $D$. Then,
\begin{equation}\label{e:1}
\left\|\frac{v}{u}\right\|_{C^{k, \alpha}(B(0, 1/2))}  \leq C ||v||_{L^{\infty}(B(0,1))}.
\end{equation}
\end{thrm}
The classical  Schauder estimates  imply that $u, v$ are $C^{k, \alpha}$ up to the boundary. Then, by the Hopf Lemma we have $u_{\nu} > 0$, and from this one can assert that the quotient $\frac{v}{u}$ is $C^{k-1,\alpha}$ up to the boundary. However, Theorem \ref{t} remarkably states that the ratio is in fact $C^{k, \alpha}$ up to the boundary. The case  $k=0$ of this result is the \emph{boundary Harnack principle} mentioned in the opening, see  \cite{CFMS} and \cite{JK}.

The purpose of this note is to generalize Theorem \ref{t} above to the heat equation and, more  generally, to linear parabolic  equations with variable coefficients. The main results are Theorem \ref{t10} and Theorem \ref{t4} below. Although our work has been strongly motivated by that of De Silva and Savin, it has nonetheless required some delicate adaptations to the parabolic setting. 

It is worth mentioning here that, besides  being an interesting regularity result in its own right, a direct application of  Theorem \ref{t} above implies $C^\infty$ smoothness of a priori $C^{1,\alpha}$ free boundaries without  the use of the hodograph transformation as in \cite{KN}, \cite{KNS}, a tool that so far has been the standard way of establishing smoothness of free boundaries  starting from $C^{1, \alpha}$. For this aspect one should see Corollary 1.2 in \cite{DS1}. Having said this, we would like to mention that the hodograph transformation in \cite{KN}, \cite{KNS} does in fact imply real-analyticity of the free boundary, which is instead not implied by Theorem \ref{t}. Nevertheless, Theorem \ref{t} provides a new perspective in the study of  Schauder theory and  free boundary  problems. Theorem \ref{t}  has also been extended to slit domains in \cite{DS2}. In the same paper such result has been used to establish smoothness of the free  boundary  in lower-dimensional obstacle problems of Signorini type near regular points. The real analyticity of the free  boundary near regular  points  in the elliptic thin obstacle  problem   has been recently  established in \cite{KPS} by using a method  based on hodograph transformation.

These recent results and  their  applications to free boundary problems motivated us to investigate their parabolic counterpart. Our main result is Theorem \ref{t10} below which constitutes the heat equation counterpart of Theorem \ref{t}. We mention that the case $k=0$ of Theorem \ref{t10} can be found in \cite{FSY}. In the present  paper we make the observation that  the ideas  in \cite{DS1} can be successfully adapted to the parabolic situation. The idea of the proof in \cite{DS1} is to approximate (after a suitable change of coordinates) $v$ by  polynomials of the type $x_n P$ by a compactness argument which uses  Schauder estimates. Then, finish by remarkable idea that $x_{n}P$ can in fact be replaced by $uP$. In our situation, as in Theorems \ref{t10} and \ref{t4}, we show that the approximating  polynomials in the space  variable in \cite{DS1} can be  suitably replaced by  appropriate  approximating  parabolic polynomials, and  one can argue in a similar manner. Modulo some delicate details, which we have tried to illustrate as much as possible. Similarly to Corollary 1.2 in \cite{DS1}, Theorem \ref{t10} implies, in particular, the $C^\infty$ smoothness in the parabolic  obstacle problem of a $C^{1,\alpha}$ free boundary near the regular points considered in Theorem 13.1 in \cite{CPS}. We note nonetheless that, analogously to the elliptic case, one can establish the space-like real analyticity of the free boundary by employing the hodograph transform as in \cite{CPS}.  It remains to be seen whether the analogue of Theorem 1.1 in \cite{DS2} is true for parabolic equations since such result would have important applications to the parabolic thin obstacle problem which was systematically studied in \cite{DGPT}. This question will be addressed in a  future  study.

This paper is organized as follows. To better demonstrate the ideas we first establish in Section \ref{r1} the  higher-regularity result in the case $k=1$ and for the heat equation. In Section \ref{s2}, still for the heat equation, we analyze the case $k\ge 2$. In Section \ref{S:vc} we extend the results of the previous sections to non-divergence form operators with variable coefficients. Section \ref{S:pob} closes the paper. In it we present an application of Theorems \ref{main1} and \ref{t10} to the parabolic  obstacle  problem studied in \cite{CPS}.

\medskip

\noindent \textbf{Acknowledgment:} The paper was finalized  during the first author's stay at the Institut Mittag-Leffler during the semester long program \emph{Homogenization and Random Phenomenon}. The first  author would like to thank the Institute and the organizers of the program for the kind hospitality and the  excellent working conditions.

\section{$H^{1+\alpha}$ regularity for the heat equation}\label{r1}
In this section we establish the case $k=1$ of the main higher regularity result which is Theorem \ref{t10} below. Our main result is Theorem \ref{main1}. In order to state it we need to introduce some preliminary notation and hypothesis.

With $x\in \Rn, t\in \R$, we will denote  by $(x,t)$ a generic point in $\R^{n+1}$. If $x_0\in \Rn$ and $t_0\in \R$ we indicate with
\begin{equation}\label{Qr}
Q_{r}(x_0, t_0)= B_{r}(x_0) \times (t_0-r^2, t_0]
\end{equation}
the parabolic cylinder ``centered" at $(x_0,t_0)$.  Given an open set $G\subset \R^{n+1}$ we say that a point $(x_0,t_0)\in \p G$ belongs to the parabolic boundary of $G$, and write $(x_0,t_0)\in \p_p G$, if for every $r>0$ the open  cylinder $B_{r}(x_0) \times (t_0-r^2, t_0)$ contains points of the complement of $G$ (notice that $(x_0,t_0)\not\in B_{r}(x_0) \times (t_0-r^2, t_0)$). Thus, for instance, when $G = \Om \times (0,T)$, then $\p_p G = (\Om \times \{0\})\bigcup (\p \Om \times [0,T])$. We denote by $SG$ the lateral boundary of $G$, see p. 13 of \cite{Li} for the relevant notions. The reader should also see p. 5 of \cite{Li} for the notion of parabolic norms and distance.
For the for the definitions of $C^{k, \alpha}$ spaces, norm and  seminorm, we refer the reader to p. 90 in \cite{GT}. We also refer to p. 46 in \cite{Li} for the relevant  notion of $H^{k+\alpha}$ spaces and p. 75 in \cite{Li} as well for the definition of domains with $H^{k+\alpha}$  boundaries.

We now consider a connected bounded open set $G\subset \{(x,t)\in \R^{n+1}\mid t \leq 1\}$, and we assume that $(0,0) \in \partial_{p} G$. We also suppose that $\partial G \cap Q_{2}(0,0)= SG \cap Q_{2}(0,0)$, and that $G \cap Q_{2}(0,0)$ is space-time $C^{1, \alpha}$ regular, i.e., there exists $f \in C^{1,\alpha}$  such that 
\[
G \cap Q_{2}(0,0) = \{(x, t)\in Q_{2}(0,0)\mid x_n > f(x',t)\}.
\]

If we  introduce  the following notations:
\[
F_{t}=\{x\in \Rn \mid  (x,t) \in G \cap Q_{2}(0,0) \},\ \ \ 
G_{t}=\{x\in \Rn \mid  (x,t) \in \p G \cap Q_{2}(0,0)\},
\]
then for each $t \leq 0$, the set $G_{t}$ is a $(n-1)$-dimensional $C^{1,\alpha}$ submanifold  which can  be equivalently  characterized in the following manner
\[
G_t=\{ (x',x_n) \mid x_n= f(x', t)\ \text{and}\ (x, t) \in Q_{2}(0,0)\}.
\]
For each $x \in G_t$ we  denote  by $\nu_{t}(x)$ the inward  unit normal to $F_t$ at $x$. In the following discussion, whenever  there i son ambiguity about the point $(x,t)$ that is being considered, we will simply write $\nu$ instead of $\nu_{t}(x)$.

We will use the notation $D' f$ when referring to the gradient of $f$ with respect to the variable $x' = (x_1, ...,x_{n-1})\in \R^{n-1}$. Without loss of generality, we may and will assume that
 \begin{equation}\label{as}
\nu= \nu_0(0) = e_n,\ \ f(0,0)=0,\  D'f(0,0)=0,\  \text{and}\  ||f||_{C^{1,\alpha}} \leq c_0,
\end{equation}
where $c_0$ is a dimensional constant which is chosen sufficiently small in such a way that $(e_n, -3/2) \in G$. This latter property can always be achieved by suitably scaling the domain as we describe later. We also normalize the function $u$ appearing in Theorem \ref{main1} below so that the following holds
\begin{equation}\label{assump}
u(e_n,- 3/2) = 1.
\end{equation}

\begin{rmrk}\label{R:universal}
Hereafter in this paper when we say that a constant is \emph{universal} we mean that it depends only on the dimension $n$, and the parameters $\alpha$ and $k$ in Theorem \ref{t10} below. We notice that in the next Theorem \ref{main1} we are taking $k =1$, and thus in this case the dependence would be only on $n$ and $\alpha$.
\end{rmrk}

The main result of this section is the following $H^{1+ \alpha}$ regularity result.

\begin{thrm}\label{main1}
Let  the domain $G\subset \R^{n+1}$ be as above and satisfy the assumption \eqref{as}. Let $u$ and $v$ be two solutions to the heat  equation in $G \cap  Q_{2}(0, 0)$, with $u, v$  vanishing  on $\partial_{p} G \cap Q_{2}(0,0)$, and suppose that $u>0$ in $G \cap Q_{2}(0,0)$ and that it satisfy the normalization \eqref{assump}. Then, for some universal $C>0$ one has
\begin{equation}\label{descon}
||\frac{v}{u}||_{H^{1+\alpha}(G \cap Q_1(0,0))} \leq C \big(||v||_{L^{\infty} (G\cap Q_{2}(0,0))}+1\big).
\end{equation}
\end{thrm}

\begin{rmrk}
It is worth mentioning here that, although in Theorem \ref{main1} we assume that the domain be  $C^{1, \alpha}$, instead of just $H^{1+\alpha}$ regular, it  is not  restrictive in its application to  free boundary problems (See Corollary \ref{fb10} below). This follows from the fact that the classical  boundary Harnack inequalities imply that for the parabolic obstacle problem studied in \cite{CPS} the Lipschitz free boundary  near a  regular  point as in Theorem 13.1 in \cite{CPS} is shown in the subsequent Theorem 14.1 in the same paper to be space-time $C^{1, \alpha}$ regular. The  hypothesis of $C^{1,\alpha}$ regularity in Theorem \ref{main1}  is only used to apply the Hopf Lemma as  in Theorem 3' in \cite{LN}. Note that the $H^{1+\alpha}$  regularity assumption on the domain does not imply that (4.6) in \cite{LN} holds, which is  precisely why we assume  that the domain be $C^{1,\alpha}$ regular.
\end{rmrk}

\begin{rmrk}
We now illustrate by an example in \cite{G} that  Theorem \ref{main1} cannot possibly be true at the base or at the corner points of a smooth cylinder.
Consider to fix the ideas $G= B(0, 1) \times (0, 1)$.  Let $u, v$ be the solutions of the heat equation in $G$ corresponding to Cauchy-Dirichlet  data $g(x, t)=t^{\alpha}$ and $h(x, t)= t^{\beta}$ respectively. Clearly,  $u$ and $v$ vanish at $t=0$. Assume that  $\beta < \alpha$ and denote  by  $K(x,t,y,s)$  the  kernel function for $G$ at the boundary such that
\[
u(x, t)= \int_0^t \int_{\partial B} K(x,t,y,s) g(y,s) d\sigma(y) ds,\ \ v(x, t)= \int_{0}^t \int_{\partial B} K(x,t,y,s) h(y,s) d\sigma(y) ds.
\]
Then, we have 
\[
\frac{v(x,t)}{u(x,t)} = \frac{\int_{0}^t \int_{\partial B} K(x,t,y,s) s^{\beta} d\sigma(y) ds}{ \int_{0}^t \int_{\partial B} K(x,t,y,s) s^{\alpha} d\sigma(y) ds} \geq \frac{1}{t^{\alpha - \beta}}  \frac{\int_0^t \int_{\partial B} K(x,t, y, s) s^{\alpha} dy ds}{ \int_0^t \int_{\partial B} K(x,t, y, s) s^{\alpha} dy ds}= \frac{1}{t^{\alpha - \beta}}.
\]
We conclude that the ratio $\frac{v}{u}$ cannot possibly be bounded as $ t\to  0^+$. This example  demonstrates that Theorem \ref{main1} is not true  in a  neighborhood of  a  point  $(x_{0},0) \in B \times \{0\}$. 

The same example can be modified  to demonstrate  that  Theorem \ref{main1} is  not true  in a neighborhood of a corner point  $(x_{0},0) \in \partial B \times \{0\}$. Let $\phi$ be a smooth function on $\partial B$ such that $\phi$ vanishes in a neighborhood of  $x_0$, and let this time $g(x, t)= \phi(x) t^{\alpha}$, $h(x, t)= \phi(x) t^{\beta}$.
As above, we obtain $\frac{v(x,t)}{u(x, t)} \geq \frac{1}{ t^{\alpha - \beta}}$,
and therefore the ratio $\frac{v}{u}$ is not  bounded in a neighborhood of $(x_0,0)$.
\end{rmrk}

Before proving Theorem \ref{main1} we make some preliminary considerations and reductions, and we establish a crucial auxiliary lemma. With $u,v$ as in Theorem \ref{main1},   by  Schauder theory (see for instance Theorem 4.27 in \cite{Li}), we have that $u, v \in H^{1+\alpha}$ up to $\partial_{p} G \cap Q_{3/2}(0,0)$, say. Moreover, by the Hopf lemma in Theorem 3' in \cite{LN}, the Schauder type estimates, \eqref{as}, the normalization condition  \eqref{assump}  and the interior Harnack inequality for parabolic equations, we have that 
\begin{equation}\label{e0}
u_{\nu}(x, t)\ge c > 0\  \text{in}\ \partial_{p} G \cap Q_1(0,0),
\end{equation}
where, we recall, $\nu$ indicates the inward  normal  at  $(x,t) \in D_{t}$.  After parabolic dilations of the domain,  i.e.,  by considering the rescaled functions 
\begin{equation}\label{d1}
 u_{r_0}(x,t)= \frac{u(r_{0} x,r_{0}^{2} t)}{r_{0}},\ \ \ \ \ v_{r_0}(x,t)= \frac{v(r_{0} x,r_{0}^{2} t)}{r_{0}},
\end{equation}
we see that $u_{r_0}, v_{r_0}$  solve the heat equation in the rescaled domain 
\[
G^{r_0}= \{(x,t)\in \R^{n+1}\mid (r_{0} x,r_{0}^2 t) \in G\}.
\]
 Moreover, $G^{r_0}$ is given near $(0,0)$ by the graph of $f_{r_0}(x',t) = \frac{f(r_0 x',r_{0}^2 t)}{r_0}$. From \eqref{d1} it is easy to see that for every $(x,t), (y,s)\in G^{r_0} \cap Q_{2}(0,0)$ one as
\[
\frac{|D_x u_{r_0}(x,t) - D_x u_{r_0}(y,s)|}{(|x-y|^2 + |t-s|)^{\alpha/2}} \le r_0^\alpha [D_x u]_{H^\alpha(G\cap Q_2(0,0))}.
\]
Therefore, given $\delta>0$, we can choose the scaling parameter $r_{0} = r_0([D_x u]_{H^\alpha(G\cap Q_2(0,0))}, \delta)>0$ in such a way that
\begin{equation}\label{udelta}
[D_x u_{r_0}]_{H^{\alpha}(G^{r_0} \cap Q_{2}(0, 0))} \leq \delta.
\end{equation}
It is then clear that for each $\delta>0$ the corresponding function $u_{r_0}$ satisfying \eqref{udelta} does depend on $\delta$. We also note that  because of \eqref{as}, if for a given $\delta>0$ the scaling parameter $r_{0}$ is suitably chosen,  and  if $u$ is  multiplied by an appropriate constant depending on $c$ in \eqref{e0},   we can ensure that the   following holds  
\begin{equation}\label{assump2}
 ||f_{r_0}||_{C^{1,\alpha}} \leq \delta, \  D_x u_{r_0}(0, 0)= e_n.
\end{equation}
We note in passing that the first inequality in \eqref{assump2}
represents a flatness assumption of the boundary of $G$ which will become important in establishing \eqref{voveru} below. Moreover,  the choice  of $\delta$  which is to be fixed  later will  be determined by Lemma \ref{imp1}  where  a compactness argument, in which we let $\delta \to 0$, is employed.  More precisely, in the proof of  Lemma \ref{imp1}, $\delta$ is  chosen  small enough  so that \eqref{deltarho} and \eqref{e7} below hold for some $\rho > 0$ universal. In conclusion, given the choice  of $\delta$  determined by Lemma \ref{imp1}, from now on to simplify the notation we will let  $u_{r_0}=u$, $v_{r_0}=v$, $G^{r_0}=G$, and establish our results for these new $u$, $v$ and $G$. 

We now introduce  the relevant notion of  approximating affine  function.

\begin{dfn}\label{D:app}
 We call  $P$ an  approximating  affine  function at $(0,0)$ if it has the  form $P(x) = \sum_{i=1}^n  a_i x_i + a_0$, with $a_n=0$.
\end{dfn}

 As  in \cite{DS1}, we have the following intermediate lemma.
\begin{lemma}\label{imp1}
Assume that for  some $r \leq 1 $ and $P$ an  approximating affine function   with $|a_i|\leq 1$, one has
\begin{equation}\label{i1}
||v - uP||_{L^{\infty}(G \cap Q_{r}(0,0))} \leq r^{2+ \alpha}.
\end{equation}
Then, there exists an  approximating affine function $\tilde P$ such that for some $C, \rho > 0$ universal, we have  
\begin{equation}\label{pptilde}
 ||P-\tilde P||_{L^{\infty}(G \cap Q_{r}(0,0))} \leq C r^{1+\alpha},
\end{equation}
and
\begin{equation}\label{vutilde}
||v - u\tilde P||_{L^{\infty}(G \cap Q_{\rho r}(0,0))} \leq (\rho r)^{2+ \alpha}.
\end{equation}
\end{lemma}

\begin{proof}
We let $\tilde G= \left\{(\frac{x}{r}, \frac{t}{r^2})\mid (x,t) \in G\right\}$, and consider the function $\tilde v$ defined in $\tilde G$ by the following equation
\begin{equation}\label{e1}
v(x,t)= u(x,t) P(x) + r^{2+\alpha} \tilde v(\frac{x}{r},\frac{t}{r^2}).
\end{equation}
Although in the next definition $\tilde u$ does not have the same meaning as $\tilde v$ in \eqref{e1} above, for later purposes we nonetheless abuse the notation and set
\begin{equation}\label{e2}
\tilde u(x,t)= \frac{u(rx,r^2 t)}{r}, \ \ \ \ \ \ (x,t)\in \tilde G.
\end{equation}

Before proceeding we pause to recall that in the reduction which precedes Lemma \ref{imp1}, given any $\delta>0$ we have chosen $r_0 >0$, depending on $\delta$, such that the function $u = u_{r_0}$  satisfy \eqref{udelta} and \eqref{assump2}. Since we also set $v = v_{r_0}$, it is clear that both $u$ and $v$ do depend on $\delta$, and therefore so do $\tilde u$ and $\tilde v$. The reader should keep this in mind when below we let $\delta \to 0$ along a sequence.

Since $u, v$ and $P$ are solutions of the heat equation, we easily obtain from \eqref{e1}
\begin{equation}\label{e3}
0= (\Delta v - v_t)(x,t)= 2 \sum_{i=1}^{n-1} a_i D_i u(x,t) + r^{\alpha}(\Delta \tilde v - \tilde v_t)(\frac{x}{r},\frac{t}{r^2}),
\end{equation}
where we have denoted by $Du = (D_1u,...,D_nu)$.
Moreover,  by  \eqref{assump2} we know that $D_iu(0,0) = 0$ for $i=1,...,n-1$, and since by \eqref{udelta} we also have 
$[Du]_{H^{\alpha}(G \cap Q_{2}(0, 0))} \leq \delta$, we conclude  
that for all $i=0, ...n-1$,
\begin{equation}\label{z1}
||D_i u||_{L^{\infty}( G \cap Q_{r}(0, 0))} \leq \delta r^{\alpha}.
\end{equation}
Therefore, by using \eqref{z1} in  \eqref{e3} we see that in $\tilde G \cap  Q_{1}(0,0)$ one has
\begin{equation}\label{inhom}
|\Delta \tilde v - \tilde v_t| \leq C \delta.
\end{equation}
Furthermore, by \eqref{e1} we obtain from \eqref{i1} the following bound  
\[
||\tilde v||_{L^{\infty}(\tilde G \cap  Q_{1}(0,0))}  \leq 1.
\]

In addition, $\tilde v$ vanishes on $\partial \tilde G \cap Q_{1}(0,0)$. Therefore, if we let $\delta \to 0$ along a sequence, we will have by compactness (by using uniform  interior $H^{2+ \alpha}$ estimates and  boundary $H^{1+\alpha}$  estimates)  that for a subsequence  $\delta \to 0$, we have $\tilde v = \tilde v(\delta)  \to v_0$ uniformly on compact subsets, where the limit function $v_0$ has the properties
\begin{equation}\label{s}
\begin{cases}
\Delta v_0 - D_t v_0 = 0\ \ \text{in}\ B_{1}^{+} \times (-1, 0],
\\
||v_0||_{L^{\infty}} \leq 1,
\\
v_0 = 0 \ \text{on}\ (\{x_n=0\} \cap B_{1}(0)) \times (-1,0].
\end{cases}
\end{equation}
Here, $B_{1}^{+}= B_{1} \cap \{x_n > 0\}$. We now denote by $V_0$ the odd reflection in $x_n$ of the function $v_0$ to the whole $B_{1}(0) \times (-1,0]$. Then, $\Delta V_0 - D_t V_0 = 0$ in $B_{1}(0) \times (-1,0]$, and $V_0$ satisfies the remaining two conditions in \eqref{s} above. In particular, from the smoothness of $V_0$ up to $(\{x_n=0\} \cap B_{1}(0)) \times (-1,0]$ and the third property in \eqref{s} above we see that $D_i V_0 =0$ for $i=1 ,..., n-1$, and $D_t V_0 = 0$ on $(\{x_n=0\} \cap B_{1}(0)) \times (-1,0]$. This gives $D_{ij} V_0 = 0$, $D_{jt} V_0 = 0$, $i, j = 1,..., n-1$, and $D_{tt}V_0 = 0$ in $(\{x_n=0\} \cap B_{1}(0)) \times (-1,0]$. Furthermore, since $V_0$ is odd in $x_n$ we also have $D_{nn} V_0(0,0) = 0$. Using the fact that the variable $t$ has weight two, from the Taylor expansion of $V_0$ at $(0,0)$ we obtain that there exists $\rho>0$ universal ($\rho \le (4C)^{-1/(1-\alpha)}$ would do), such that 
\begin{equation}\label{e4}
||v_0- x_n Q_0||_{L^{\infty}(B_{\rho}^{+}(0) \times [-\rho^2, 0])} \leq C \rho^3 \leq \frac{1}{4} \rho^{2+ \alpha},
\end{equation}
where $Q_0(x) = \Sigma_{i=1}^{n-1}  q_i x_i +q_0$ is an affine function which is approximating (note that $q_n = 0$ is a consequence of $D_{nn}V_0(0,0) = 0$), with 
\begin{equation}\label{z10}
|q_i| \leq C_1
\end{equation}
for some  $C_1>0$ universal. In particular,  the product $x_n Q_0(x)$ is harmonic. We now fix such a universal $\rho>0$. Since $\tilde v = \tilde v(\delta) \to v_0$ uniformly on compact sets as $\delta \to 0$ on a subsequence, by compactness we see that for $\delta$ sufficiently small we have,
\begin{equation}\label{deltarho}
||\tilde v - v_0||_{L^{\infty}(\tilde G \cap Q_{\rho}(0,0) )} \le   \frac{1}{4} \rho^{2+ \alpha}.
\end{equation}
From \eqref{deltarho} and \eqref{e4} we thus conclude that for $\delta$ sufficiently small
\begin{align}\label{e5}
||\tilde v - x_n Q_0||_{L^{\infty}(\tilde G \cap Q_{\rho}(0,0))} & \leq 
||\tilde v - v_0||_{L^{\infty}(\tilde G \cap Q_{\rho}(0,0))} + ||v_0 - x_n Q_0||_{L^{\infty}(\tilde G\cap Q_{\rho}(0,0))}
\\
& \le \frac{1}{4} \rho^{2+ \alpha} + \frac{1}{4} \rho^{2+ \alpha}  = \frac{1}{2} \rho^{2+ \alpha}.
\notag
\end{align}
Moreover, from \eqref{udelta} and \eqref{assump2} we have that
\begin{equation}\label{e6}
||\tilde u - x_n||_{L^{\infty}(\tilde G \cap Q_{\rho}(0,0))} \leq \delta.
\end{equation}
Thus, from the triangle  inequality, if with $C_1$ as in \eqref{z10} we further restrict $\delta$ in dependence of $\rho$ so that $n C_1 \delta \leq \frac{1}{2} \rho^{2+\alpha}$ , we have 
\begin{equation}\label{e7}
||\tilde v - \tilde u Q_0||_{L^{\infty}(\tilde G \cap Q_{\rho}(0,0))} \leq \rho^{2+ \alpha}.
\end{equation}
The conclusion now follows  by taking $\tilde P(x)= P(x) + r^{1+\alpha} Q_0(\frac{x}{r})$ and by rewriting $\tilde v, \tilde u$ in terms of $v, u$. We notice explicitly that \eqref{pptilde} follows from the definition of $\tilde P$ and from the fact that \eqref{z10} gives
\[
||Q_0(\frac{.}{r})||_{L^{\infty}(G \cap Q_{r}(0, 0))} \le n C.
\]
By using \eqref{e1}, \eqref{e2} and the definition of $\tilde P$, we finally obtain
\begin{align*}
||v - u\tilde P||_{L^{\infty}(G \cap Q_{\rho r}(0,0))}\ & =\  ||r^{2+\alpha} \tilde v(\frac{.}{r},\frac{.}{r^2}) - r^{1+\alpha} u Q_0(\frac{.}{r})||_{L^{\infty}(G \cap Q_{\rho r}(0,0))}
\\
&  =\ r^{2+\alpha} ||\tilde v(\frac{.}{r},\frac{.}{r^2}) -  \tilde u(\frac{.}{r},\frac{.}{r^2})  Q_0(\frac{.}{r})||_{L^{\infty}(G \cap Q_{\rho r}(0,0))}
\\
& \le\ (\rho r)^{2+\alpha},
\end{align*}
where in the last inequality we have used \eqref{e7}. This proves \eqref{vutilde}, thus completing the proof.

\end{proof}

With this lemma in hand, one can establish Theorem \ref{main1} by arguing as in \cite{DS1}. We nevertheless provide the details for the sake of completeness.

\begin{proof}[Proof of Theorem \ref{main1}]
Now  suppose there exists   an affine function $P$ for which the hypothesis of Lemma \ref{imp1} holds for some $r>0$.  Then with $\tilde v, \tilde u$ as in the proof of Lemma \ref{imp1}, we first  note that   by $H^{1+\alpha}$ estimates up to the boundary as in Theorem 4.27 in \cite{Li} and from the fact that $\tilde v$   vanishes on $\tilde G \cap Q_{2}(0, 0)$, we have for every $(x,t) \in \tilde G \cap Q_{1/2}(0, 0)$
\begin{equation}\label{f1}
|\tilde v(x, t)|  \leq C_1 d(x,\tilde G_{t}) ,
\end{equation}
where $\tilde G_{t}= \{x\mid (x,t) \in \partial  \tilde G \cap \overline{Q}_{1}(0,0)\}$, and $d(x,\tilde G_t)$ is the Euclidean distance  of the point $x \in \Rn$ from $\tilde G_t$. Note that \eqref{f1} is a reformulation of the Lipschitz estimate in the  $x$ variable at the boundary.  Moreover, from \eqref{e0} and the definition \eqref{e2} of $\tilde u$ we find
\[
\tilde u_{\nu} \geq c > 0,
\]
where $c$ is the universal constant in \eqref{e0}. From this inequality we easily obtain
\begin{equation}\label{f2}
\tilde u(x,t) \geq  C_2 d(x, \tilde G_{t}).
\end{equation}
The estimates \eqref{f1} and \eqref{f2} imply that   
\begin{equation}\label{tildes}
|\tilde v| \leq C\tilde u\ \ \ \ \ \text{in}\  \tilde G \cap  Q_{1/2}(0, 0).
\end{equation}

We now see how the desired conclusion \eqref{descon} above can be derived from \eqref{tildes} and from Lemma \ref{imp1}.

First, by rewriting $\tilde v, \tilde u$ in terms of $v$ and $u$ using \eqref{e1}, \eqref{e2}, we obtain as a direct consequence of \eqref{tildes} 
\begin{equation}\label{f3}
|v -  uP| \leq C u r^{1+ \alpha} \  \text{in}\   G \cap  Q_{r/2}(0,0).
\end{equation}
 Moreover, since by \eqref{f2} the function $\tilde u$ is bounded away from zero from below in $ \tilde G \cap Q_{1/4}(\frac{1}{2} e_n,0))$, from $H^{1+\alpha}$ estimates for $\tilde v$ and the fact that this function satisfies \eqref{inhom}, we also have
\begin{equation}\label{e11}
||\frac{\tilde v}{\tilde u}||_{H^{1+\alpha}( \tilde G \cap Q_{1/4}(\frac{1}{2} e_n,0))} \leq C \big(||v||_{L^\infty(G\cap Q_2(0,0))} + 1\big).
\end{equation}
In \eqref{e11} we have used the fact that, because of \eqref{assump2} and the definition of $\tilde G$, at each  time  level $t \in (-1,0]$ the set $\tilde K_t=\{x\mid (x,t)\in \tilde G \cap Q_{1/4}( \frac{1}{2} e_n,0)\}$ is at a Euclidean distance from $\tilde G_t$ which is bounded  below  by $C_5$, for  some $C_5>0$  universal.
In addition, the identity 
\begin{equation}\label{e25}
\frac{v(x,t)}{u(x,t)}= P(x) + r^{1+\alpha} \frac{\tilde v(\frac{x}{r},\frac{t}{r^2})}{\tilde u(\frac{x}{r},\frac{t}{r^2})},
\end{equation}
which follows from \eqref{e1}, \eqref{e2},
implies that 
\begin{equation}\label{e26}
[D_x (\frac{v}{u})]_{H^{\alpha}( G \cap Q_{r/4}( \frac{r}{2} e_n,0))}= [D_x (\frac{\tilde v}{\tilde u})]_{H^{\alpha}( \tilde G \cap Q_{1/4}( \frac{1}{2} e_n,0))}
\end{equation}
and
\begin{equation}\label{e27}
< \frac{v}{u} >_{1+\alpha; G \cap Q_{r/4}( \frac{r}{2} e_n, 0))}= < \frac{\tilde v}{\tilde u}>_{1+ \alpha; \tilde G \cap Q_{1/4}( \frac{1}{2} e_n, 0))}.
\end{equation}
(See page  46 in \cite{Li} for the definition of the seminorm  $<w>_{1+\alpha; \Om}$ which  corresponds  to the $\frac{1+\alpha}{2}$-H\"older seminorm  of $w$ in $t$).

Summing up, \eqref{e11}, \eqref{e25}, \eqref{e26} and \eqref{e27} imply the following estimate
\begin{equation}\label{voveru}
||\frac{v}{u}||_{H^{1+\alpha}(G \cap Q_{r/4}(\frac{r}{2} e_n,0))} \leq C \big(||v||_{L^\infty(G\cap Q_2(0,0))} + 1\big).
\end{equation}
We  would like to mention that the region $G \cap Q_{r/4}( \frac{r}{2} e_n, 0)$ is at a parabolic distance  proportional to $r$  from $\partial G$, a fact that follows from the first inequality in \eqref{assump2} which represents a flatness assumption on the boundary.

Multiplying $v$ by a suitable constant, we may now assume that the hypothesis of the Lemma \ref{imp1} is satisfied for $r_0$ small and $P=0$. Given any $0<r<1$, with $\rho$ fixed as in Lemma \ref{imp1} we choose $k\in \mathbb N$ such that $\rho^{k+1} r_0 \le r < \rho^k r_0$. We thus apply Lemma \ref{imp1} iteratively, i.e., first for $r_0$, then for $\rho r_0, \rho^2 r_0$  and so on. We finally obtain a limiting affine function $P_{0}$ such that
\begin{equation}\label{tilde1}
||v -uP_0||_{L^{\infty}(G \cap Q_{r}(0,0))} \leq C r^{2+ \alpha}\ \text{for}\  r \leq r_0.
\end{equation}
We note explicitly that 
\[
P_0(x) = \sum_{i=1}^\infty  (\rho^{i-1}r_0)^{1+\alpha} Q_i(\frac{x}{ \rho^{i-1}r_0}),
\]
 where $Q_i$ is the affine function obtained after the $i$-th application of Lemma \ref{imp1} in the iteration argument.
Given  that \eqref{tilde1} holds with $P=P_0$ for every $r\leq r_0$, therefore for any given $r \leq r_0$,  we can   now repeat the arguments  which lead  to  \eqref{f3} with $P=P_0$ and consequently   obtain for all $(x,t) \in G \cap Q_{1}(0,0)$ 
\begin{equation}\label{e15}
\left|\frac{v}{u}(x, t) -P_0(x)\right| \leq C (|x|^2 + |t|)^{\frac{1+\alpha}{2}}.
\end{equation}
This implies the $H^{1+\alpha}$-regularity at the  boundary point $(0,0)$. Combining \eqref{e15} with \eqref{voveru}, the desired conclusion follows  by arguing  for instance  as in the proof of Proposition 4.13 in \cite{CC}. We mention that, although the latter result is for the elliptic setting, the same argument goes through in the parabolic setting when the Euclidean distance in $\Rn$ is replaced  by the  parabolic distance in $\R^{n+1}$.

\end{proof}

\section{Higher regularity}\label{s2}

In this section we  establish the case $k>1$ of Theorem \ref{main1}. Our main result is the following.

\begin{thrm}\label{t10}
Let $G$ be of class  $H^{k+\alpha}$ in $G \cap Q_{2}(0,0)$, the other assumptions on $G$ being as in Section \ref{r1}. Let  $u$ and $v$ be two solutions of the heat  equation in $G \cap Q_{2}(0,0)$ such that $u, v$  vanish  on $\partial_{p} G \cap Q_{2}(0,0)$. Also, suppose that $u>0$ in $G \cap Q_{2}(0,0) $  and assume that it satisfy the normalization \eqref{assump}. Then, for some $C>0$ universal one has
\begin{equation}\label{mr}
||\frac{v}{u}||_{H^{k+\alpha}(G \cap Q_1(0,0))} \leq C \left(||v||_{L^{\infty} (G\cap Q_{2}(0,0))} + 1\right).
\end{equation}
\end{thrm}

Before  proving the theorem  above, we   introduce some additional  notations. Henceforth, we let $\mathbb N_0 = \mathbb N \cup \{0\}$. Given a multi-index $m = (m_1, .... m_n)\in \mathbb N_0^n$, for $x = (x_1,...,x_n)\in \Rn$ we denote by $x^m= x_{1}^{m_{1}}\ ...\ x_{n}^{m_{n}}$ the monomial of degree $|m| = m_1+...+m_n$. Henceforth, for $i = 1,...,n$, the notation $\overline{i}$ will indicate the multi-index in $\mathbb N_0^n$ which has $1$ in the $i$-th position and $0$ elsewhere.

\begin{dfn}\label{D:parpol}
By a parabolic polynomial of degree $k\in \mathbb N_0$ we mean an expression  of the form 
 \[
 P(x,t) = \sum_{0\le |m| + 2\ell \le k} a_{m,\ell} x^m t^\ell,
 \]
where $a_{m,\ell} \in \R$. Given such a $P$ we define its norm as follows
\[
||P|| = \max_{0\le |m| + 2\ell \le k}\ |a_{m,\ell}|.
\]
\end{dfn}

Following \cite{DS1}, see also Chapter 1 in \cite{CC}, we introduce the following definition.

\begin{dfn}\label{D:holderclasses}
We say that a function $f$ is pointwise $H^{k+\alpha}$ at $(0,0)$ if there exists a parabolic polynomial $P(x,t)$ of degree $k$ such that $f(x,t) = P(x,t) + O\left((|x|^2+|t|)^{\frac{k+\alpha}{2}})\right)$. We indicate this with $f\in H^{k+\alpha}(0,0)$ and we denote with $||f||_{H^{k+\alpha}(0,0)}$ the smallest $M>0$ such that $||P||\le M$ and $|f(x,t) - P(x,t)|\le M (|x|^2+|t|)^{\frac{k+\alpha}{2}})$ for $t\le 0$. We say that $f$ is pointwise $H^{k+\alpha}$ at $(x_0,t_0)$ if the function $h(x,t) = f(x+x_0,t+t_0)$ is pointwise $H^{k+\alpha}$ at $(0,0)$, and we indicate this with $f\in H^{k+\alpha}(x_0,t_0)$. Finally, given a bounded set $G\subset \R^{n+1}$ we say that $f\in H^{k+\alpha}(G)$ if 
\[
||f||_{H^{k+\alpha}(G)}\overset{def}{=} \underset{(x_0,t_0)\in G}{\sup} ||f||_{H^{k+\alpha}(x_0,t_0)} <\infty.
\] 
\end{dfn}

We note that the class $H^{k+\alpha}(G)$ coincides with that defined on p. 46 in \cite{Li}.

Now  with $u,v$ as in Theorem \ref{t10}, it follows from the Schauder theory (see Chapters 4 and 5 in \cite{Li})  that  $u, v$ are in $H^{k+\alpha}$ up to $\overline{G \cap Q_{2}(0,0)}$. As in Section \ref{r1}, we   assume as well that
\begin{equation}\label{as1}
\nu= \nu_0(0) = e_n,\ \ f(0,0)=0,\  D'f(0,0)=0,\  \text{and}\  ||f||_{H^{k+ \alpha}} \leq c_0,
\end{equation}
where $c_0$ is a dimensional constant which is chosen sufficiently small in such a way that $(e_n, -3/2) \in G$. This latter property can always be achieved by suitably scaling the domain as  in \eqref{d1}-\eqref{assump2}.  Furthermore, we assume that
\begin{equation}\label{h1}
Du(0, 0)=e_n,\ \ \ ||u-x_n||_{H^{k + \alpha}(G \cap Q_{2}(0,0))} \leq \delta,\ \ \ ||f||_{H^{k +  \alpha}} \leq \delta.
\end{equation}
The second inequality in \eqref{h1} can be seen  as follows. Since $u(0,0)=0$ there is a parabolic polynomial  $P_{u}$, of  degree at most $k$, such that for all $(x,t) \in G \cap Q_{3/2}(0,0)$ one has
\begin{equation}\label{h100}
|u(x,t) - P_u(x,t)| \leq  C (|x|^2 + |t|)^{\frac{k+\alpha}{2}}.
\end{equation}
Note that  such a  polynomial  corresponds to the weighted Taylor expansion of  order $k$  of $u$ at $(0,0)$ in which derivatives  in $t$ are assigned the weighted order $2$. If we  let $u_{r_0}$ be as in \eqref{d1}, then  with $Du(0,0)= e_n$ we  have that 
\begin{equation}
|u_{r_0}(x,t)  -x_n| \leq  \frac{r_{0}^2}{r_0}| P_1(r_0 x, r_{0}^2 t)| = r_0 | P_1(r_0 x, r_{0}^2 t)|,
\end{equation}
where $P_1(x,t)= P_u(x, t)- x_n$. Therefore,  by choosing $r_0$  sufficiently  small the second inequality in \eqref{h1} can be ensured with our new $u=u_{r_0}$ and $G=G^{r_0}$, following computations similar to those in Section \ref{r1}.

Let $P$ be a given parabolic polynomial of degree $k$. By using the fact that $\Delta u - u_t= 0$, we have  that
\begin{equation}\label{e19}
(\Delta - D_t)(uP)= 2 <Du,DP> + u (\Delta P   - P_{t}).
\end{equation}
As a first step we write a formula for \eqref{e19} when $P(x,t) =   x^m t^\ell$, for a  given  multi-index $m\in \mathbb N_0^n$ and $\ell\in \mathbb N_0$. This is \eqref{e20} below. To derive such formula we first notice that from \eqref{h1} it follows that $u= x_{n} + w$ where $w$ is of parabolic order $\geq 2$ and $||w||_{H^{k-1, \alpha}(G \cap Q_{2}(0,0))} \leq \delta$. On the other hand, \eqref{h100} gives
$u(x, t)= P_u(x, t)+ z(x, t)$, where $P_u$ is a polynomial of degree at most $k$. Combining these two facts, and letting $P_1(x,t) = P_u(x,t) - x_n$, we can thus write 
\begin{equation}\label{uxn}
u(x,t) = x_n + P_1(x,t) + z(x,t),
\end{equation}
where, we note explicitly, the polynomial $P_1$ is of degree at least two, provided it is nonzero. We thus have
\[
Du(x,t)= e_n + DP_{1}(x,t) + Dz(x,t).
\]
This gives
\begin{align*}
2<Du,DP> + u(\Delta P -P_t) =\ & 2 D_n P + x_n(\Delta P - P_t) + 2 <DP_1,DP>
\\
+\ & 2 <Dz,DP> + (P_1 + z)(\Delta P -P_t).
\end{align*}
Keeping in mind that 
\[
D_n P = m_n x^{m-\overline n} t^\ell,\ \ \ x_n D_{nn} P = m_n(m_n-1) x^{m-\overline n} t^\ell,
\]
we find
\begin{align*}
(\Delta - D_t)(uP)\ =\ & m_n (m_{n}+1)x^{m-\overline{n}} t^\ell + \sum_{i \neq n} m_i (m_i - 1) x^{m - 2 \overline{i} + \overline{n}} t^\ell  - \ell x^{m+ \overline{n}} t^{\ell-1}
\\
+\ & \big\{2 <DP_1,DP> + P_1(\Delta P -P_t)\big\} +  \big\{2 <Dz,DP> +  z (\Delta P -P_t)\big\}.
\end{align*}
It is now easy to see that
\[
2 <DP_1,DP> + P_1(\Delta P -P_t) = \sum_{0\le |q|+2\kappa\le |m|+2\ell + k - 2} c^{m ,\ell}_{q,\kappa} x^q t^\kappa.
\]
We note explicitly that, since $P_1$ is of degree at least $2$, one has that  $c^{m, \ell}_{q, \kappa} \neq 0$ only when $|m| + 2 \ell \leq |q| + 2 \kappa$. However, for later purposes we find it convenient to isolate from the right-hand side of the latter equation a parabolic polynomial which is of degree at most $k-1$, i.e., we decompose
\begin{align*}
\sum_{0\le |q|+2\kappa\le |m|+2\ell + k - 2} c^{m ,\ell}_{q,\kappa} x^q t^\kappa = & \sum_{|m| + 2\ell \leq |q| + 2\kappa \leq k-1} c^{m ,\ell}_{q,\kappa} x^q t^\kappa
\\
+ &  \sum_{k\le |q|+2\kappa\le |m|+2\ell + k - 2} c^{m ,\ell}_{q,\kappa} x^q t^\kappa. 
\end{align*}
If we define
\[
w_{m,\ell}(x,t) = 2 <Dz,DP> +  z (\Delta P -P_t) + \sum_{k\le |q|+2\kappa\le |m|+2\ell + k - 2} c^{m ,\ell}_{q,\kappa} x^q t^\kappa,
\]
then we conclude that
\begin{align}\label{e20}
(\Delta - D_t)(uP)\ =\ & m_n (m_{n}+1)x^{m-\overline{n}} t^\ell + \sum_{i \neq n} m_i (m_i - 1) x^{m - 2 \overline{i} + \overline{n}} t^\ell  - \ell x^{m+ \overline{n}} t^{\ell-1}
\\
 +\ & \sum_{|m| + 2\ell \leq |q| + 2\kappa \leq k-1} c^{m ,\ell}_{q,\kappa} x^q t^\kappa +  w_{m,\ell}(x,t),
\notag
\end{align}
where we have $|c^{m,\ell}_{q,\kappa}| \le C \delta$, a fact which follows from $||w||_{H^{k-1, \alpha}(G \cap Q_{2}(0,0))} \leq \delta$. 
Moreover,  again from \eqref{h1} we have for all $(x,t) \in G \cap Q_{1}(0,0)$  that
\begin{equation}\label{calc}
|w_{m,\ell}(x,t)| \leq C \delta ((|x|^2 + |t|)^{\frac{k-1+ \alpha}{2}}).
\end{equation}
In \eqref{calc} we have crucially used the fact that $|Dz(x,t)| \le C \delta (|x|^2 + |t|)^{\frac{k-1+ \alpha}{2}})$ for all $(x,t)\in G\cap Q_1(0,0)$.

We now turn to determining a suitable expression for the right-hand side in \eqref{e19} in the case in which 
\[
P(x,t) = \sum_{0\le |m|+2\ell\le k} a_{m,\ell} x^m t^\ell
\]
 is a general parabolic polynomial of degree $k$. In such case, we obtain from \eqref{e20} 
\begin{equation}\label{PRw}
(\Delta - D_t) (uP)= R(x) + w(x), 
\end{equation}
where $R$ is the parabolic polynomial of degree $k-1$ given by
\begin{align}\label{R}
R(x,t)=  \sum_{0\le |q| + 2\kappa \leq k-1} d_{q,\kappa} x^{q} t^{\kappa} & \ = \sum_{0\le |m|+2\ell\le k} a_{m,\ell} \bigg\{m_n (m_{n}+1)x^{m-\overline{n}} t^\ell + \sum_{i \neq n} m_i (m_i - 1) x^{m - 2 \overline{i} + \overline{n}} t^\ell
\\
 & - \ \ell x^{m+ \overline{n}} t^{\ell-1}
 + \sum_{|m| + 2\ell \leq |q| + 2\kappa \leq k-1} c^{m ,\ell}_{q,\kappa} x^q t^\kappa\bigg\},
\notag
\end{align}
and
\[
w(x)= \sum_{0\le |m|+2\ell\le k}  a_{m,\ell} w_{m,\ell} (x).
\]
Note that \eqref{R} gives
\begin{equation}\label{f5}
(q_n +1) (q_n +2) a_{q+ \overline{n}, \kappa} + \sum_{i \neq n} (q_i + 1) (q_i + 2) a_{q+ 2 \overline{i} - \overline{n}, \kappa}   - (\kappa+1) a_{q- \overline{n}, \kappa+1} + c^{m,\ell} _{q,\kappa}a_{m,\ell}= d_{q,\kappa}.
\end{equation}
Given a parabolic polynomial of degree $k-1$ such as $R(x)=  \sum_{0\le |q| + 2\kappa \leq k-1} d_{q,\kappa} x^{q} t^{\kappa}$, our objective is finding a parabolic polynomial of degree $k$, $P(x,t) = \sum_{0\le |m|+2\ell\le k} a_{m,\ell} x^m t^\ell$, such that  \eqref{PRw} hold (in particular, we will be interested in this section in the case when $R\equiv 0$, see Definition \ref{D:R} below). This will be possible if, given constants $d_{q,\kappa}$, we can solve the linear system \eqref{f5} above for the unknowns $a_{m,\ell}$.
Notice that one can think of \eqref{f5} as an equation where  $a_{q+ \overline{n}, \kappa}$ is a linear combination of $d_{q, \kappa}$'s and $a_{m,\ell}$'s such that either $|m| + 2\ell <|q| + 2\kappa+ 1$, or when $|m| + 2\ell = q + 2\kappa + 1 $, then $m_n < q_{n}+ 1$. It thus follows that we can solve \eqref{f5} if we arbitrarily assign all the coefficients $a_{m,\ell}$ when $m=(m_1,...,m_n)$ is a multi-index having $m_n = 0$. In this respect we emphasize that the crucial fact which makes this claim possible is that the third term in the left-hand side of \eqref{f5}, namely, the one which comes from differentiating in the time variable $t$, has the same degree as the first two terms. We define the order of a  coefficient  $a_{m,\ell}$  as  $|m| + 2\ell$, i.e., the weighted degree of the corresponding monomial $x^m t^\ell$.   

To verify the above claim we briefly describe the  procedure of determining  the coefficients. First of all, $a_{(0,...,0),0}$, which is the unique coefficient  of order $0$, is assigned arbitrarily since it trivially satisfies the requirement $m_n=0$. We proceed by a double induction. Suppose we know all coefficients  $a_{m,\ell}$  up to order  $p$. Given a coefficient $a_{m_0,\ell_0}$ of order $p+1$, i.e., $|m_0| + 2 \ell_0=p+1$, let $(m_0)_{n}$ denote  the entry  at the $n$-th  position of the corresponding multi-index $(m_0,\ell_0)$. Clearly, $0 \leq (m_0)_{n} \leq p+1$. We  determine  all coefficients  of order $p+1$  by induction on $(m_0)_n$ . First, all coefficients $a_{m_0,\ell_0}$ of order  $p+1$ with  $(m_0)_{n}=0$ are arbitrarily assigned. Suppose now all coefficients $a_{m_0,\ell_0}$  of order  $p+1$ with $(m_0)_n \leq \chi$ are known. Then,  given a coefficient $a_{m_1, \ell_1}$ of order $p+1$ with $(m_1)_n= \chi+1$,  we have from \eqref{f5} that such a  coefficient  is expressible in  terms of lower order coefficients $a_{m,\ell}$, which are  already  determined,  and coefficients $a_{m_0,\ell_0}$ of order  $p+1$  such that $(m_0)_{n} \leq \chi$ and  which are  supposed to be known by the induction hypothesis on $(m_0)_n$.  Therefore,  all coefficients $a_{m_0, \ell_0}$ of order $p+1$ with $(m_0)_{n}=\chi +1$ can be  determined once all coefficients $a_{m_0,\ell_0}$  up to order $p+1$ with $(m_0)_n \leq \chi$ are known.  In this manner, all coefficients up to order $p+1$ can be determined once all coefficients  of up to  order $p$ are known. The claim thus follows. 

For a  fixed  $k$ as in Theorem \ref{t10} we  now  consider parabolic polynomials of degree $\leq k$. We next introduce a notion which generalizes to the case $k\ge 2$ that in Definition \ref{D:app} above.

\begin{dfn}\label{D:R}
We say that $P(x,t) = \sum_{0\le |m|+2\ell\le k} a_{m,\ell} x^m t^\ell
$ is  an approximating parabolic polynomial of order $k$  for $\frac{v}{u}$ at $(0,0)$ if we have $R(x,t) \equiv 0$ in the representation \eqref{PRw} above. This is equivalent to the fact that the $a_{m,\ell}$ satisfy \eqref{f5} with $d_{q,\kappa} = 0$.
\end{dfn}

\begin{rmrk}\label{R:ap}
The motivation for Definition \ref{D:R} comes form the fact that, when $P$ is approximating, then from \eqref{PRw} we obtain $H(uP)=w$, with $w$ of weighted order $k-1+ \alpha$. This is crucially used in \eqref{e23} below since it allows us to cancel the term $r^{k-1+\alpha}$ from both sides of the equation. We have in fact, see \eqref{calc}, $w(x,t)= (|x|^2+ |t|)^{\frac{k-1+\alpha}{2}} w_1(x,t)$, where $||w_1(x,t)|| \leq \delta$. Note that when $k=1$, and therefore $P(x,t) = a_0 + \sum_{i=1}^n a_i x_i$, this notion is equivalent to saying that $x_n P$ is caloric, which is in turn equivalent to the condition $a_n=0$. In that case from \eqref{z1} we see $H(uP) \leq C \delta  r^\alpha$, which in the case $k=1$ is the $w$ of order $1-1+\alpha = \alpha$.
We also mention that in the case of the variable coefficient operators treated in Section \ref{S:vc} we will no longer impose the condition $R=0$ in the definition of approximating polynomial. This is so because of a nonzero right-hand side. 
\end{rmrk}

The proof of  Theorem \ref{t10}  follows the same lines as that of Theorem \ref{main1} once the following lemma is established.

\begin{lemma}\label{imp2}
Assume that for  some $r \leq 1 $ and $P$ an  approximating polynomial of order $k$ for $\frac{v}{u}$ at $(0,0)$   with $||P||\leq 1$, one has
\begin{equation}\label{i1enough}
||v - uP||_{L^{\infty}(G \cap Q_{r}(0,0))} \leq r^{k+1+ \alpha}.
\end{equation}
Then, there exists an  approximating polynomial  $\tilde P$ of order $k$ such that for some $C, \rho > 0$ universal, we have  
\begin{equation}\label{pptilde2}
 ||P-\tilde P||_{L^{\infty}(G \cap Q_{r}(0,0))} \leq C r^{k+\alpha},
\end{equation}
and
\begin{equation}\label{vutilde2}
||v - u\tilde P||_{L^{\infty}(G \cap Q_{\rho r}(0,0))} \leq (\rho r)^{k +1+ \alpha}.
\end{equation}
\end{lemma}

\begin{proof}
With $\tilde G$ as in the proof of Lemma \ref{imp1}, we define $\tilde v(x,t)$ by the equation
\begin{equation}\label{e21}
v(x,t)= u(x, t) P(x,t) + r^{k + 1 +\alpha} \tilde v (\frac{x}{r},\frac{t}{r^2}),
\end{equation}
and we let 
\begin{equation}\label{tildeub}
\tilde u(x,t) = \frac{u(rx,r^2 t)}{r}.
\end{equation}
Since $v$ is a solution of the heat equation we have 
\begin{equation*}
0=\Delta v - v_t = (\Delta - D_t)(uP) + r^{k-1+\alpha}(\Delta \tilde v  - D_t\tilde v)(\frac{x}{r},\frac{t}{r^2}).
\end{equation*}
Using  the fact that $P$ is an approximating polynomial of order $k$ and \eqref{PRw}, we conclude
\begin{equation}\label{e23}
r^{k-1+ \alpha} (\Delta \tilde v  - D_t \tilde v)(\frac{x}{r},\frac{t}{r^2})  =  - w(x,t).
\end{equation}
We emphasize the crucial role played in the latter equation by the property of $P$ being an approximating polynomial, see Remark \ref{R:ap} above. By \eqref{calc} we conclude that in $\tilde G\cap Q_1(0,0)$ one has
\begin{equation*}
(\Delta - D_t)\tilde v=h,
\end{equation*}
where 
\begin{equation*}
h \in H^{k-2+\alpha}(\tilde G\cap Q_1(0,0)) \  \ \ \ \text{and}\ \ \ \ |h| \leq  C\delta.
\end{equation*}

In addition, $\tilde v$ vanishes on $\partial \tilde G \cap Q_{1}(0,0)$. Therefore, if we let $\delta \to 0$ along a sequence, we will have by compactness (by using uniform  interior $H^{k+1+ \alpha}$ estimates and  boundary $H^{k+\alpha}$  estimates)  that for a subsequence  $\delta \to 0$, we have $\tilde v = \tilde v(\delta)  \to v_0$ uniformly on compact subsets, where the limit function $v_0$ has the properties
\begin{equation}\label{s1}
\begin{cases}
\Delta v_0 - D_t v_0 = 0\ \ \text{in}\ B_{1}^{+} \times (-1, 0],
\\
||v_0||_{L^{\infty}} \leq 1,
\\
v_0 = 0 \ \text{on}\ (\{x_n=0\} \cap B_{1}(0)) \times (-1,0].
\end{cases}
\end{equation}

Proceeding as in the proof of Lemma  \ref{imp1} we can find  a polynomial  $Q_0$ of degree $k$ such that $x_n Q_0$ solves the heat equation, and 
\begin{equation}\label{z10bis}
|| v_0 - x_n Q_0||_{L^{\infty}(B_{\rho}^{+} \times (-\rho^2, 0])} \leq C \rho^{k+2} \leq \frac{1}{4} \rho^{k+1 + \alpha},
\end{equation}
for some $C, \rho > 0$ universal.  This follows  for instance  by odd reflection of  $v_{0}$  across $x_n =0$  (which provides again a solution to the heat equation), and by applying Lemma 1.1 in \cite{AV} to the extended function. Then, again for  $\delta>0$ sufficiently small, we have
\begin{equation}\label{z11}
||\tilde v - v_0||_{L^{\infty}(\tilde G \cap Q_{\rho}(0,0))} \le  \frac{1}{4} \rho^{k+1+ \alpha}.
\end{equation}
Therefore, from \eqref{z10bis}, \eqref{z11} and the triangle inequality we obtain,
\begin{equation}\label{f10}
||\tilde v- x_n Q_0||_{L^{\infty}(\tilde G \cap Q_{\rho}(0,0))} \leq\frac{1}{2} \rho^{k+1 + \alpha}.
\end{equation}
Now, by using the second inequality in \eqref{h1} we conclude that, if the choice of  $\delta$ is further  restricted in such a way that $C \delta \leq \frac{1}{4} \rho^{k+1+\alpha}$, for some  $C$ which depends on $||Q_0||$ and hence is universal, then 
\begin{equation*}
||\tilde v- \tilde u Q_0||_{L^{\infty}(\tilde G \cap Q_{\rho}(0,0))} \leq \frac{3}{4} \rho^{k+1 + \alpha}.
\end{equation*}
Therefore,  by rewriting $\tilde v$ and $\tilde u$ in terms of $u$  and $v$ we have that
\begin{equation}\label{f11}
||v -  u \big(P + r^{k+ \alpha} Q_0 (\frac{.}{r},\frac{.}{r^2})\big)||_{L^{\infty}(\tilde G \cap Q_{\rho r}(0,0))}  \leq \frac{3}{4}(\rho r)^{k+1 + \alpha}.
\end{equation}
At this point, \eqref{f11} would complete the proof of the lemma if we knew that the parabolic polynomial  $P + r^{k+ \alpha} Q_0(\frac .r, \frac{.}{r^2})$ is approximating for $\frac{v}{u}$. Unfortunately, this is not necessarily the case and we need to further modify $Q_0$ to some other polynomial $\tilde Q$, without essentially modifying the estimate \eqref{f11}. 

Let us write $Q_0(x,t) = \sum_{0\le |q|+2\kappa\le k} b_{q,\kappa} x^q t^\kappa$ and $\tilde Q(x,t) = \sum_{0\le |q|+2\kappa\le k} \tilde b_{q,\kappa} x^q t^\kappa$. In view of \eqref{f5}, and  by using the fact that $P(x,t) = \sum_{0\le |m|+2\ell\le k} a_{m,\ell} x^m t^\ell$ is an approximating  polynomial  for $\frac{v}{u}$ of degree $k$, we obtain that, in order for 
\begin{equation}\label{tildepbis}
\tilde P(x,t) \overset{def}{=} P(x,t) + r^{k+\alpha} \tilde Q(\frac{x}{r},\frac{t}{r^2})
\end{equation}
 to be an approximating  polynomial  for $\frac{v}{u}$ of degree $k$, the coefficients $\tilde b_{q, \kappa}$ of  $\tilde Q$ should satisfy
\begin{equation}\label{f12}
(q_n +1) (q_n +2) \tilde b_{q+ \overline{n}, \kappa}  + \sum_{i \neq n} (q_i + 1) (q_i + 2) \tilde b_{q+ 2 \overline{i} - \overline n, \kappa} - (\kappa +1) \tilde b_{q- \overline{n}, \kappa+1} + \tilde c^{m,\ell} _{q,\kappa} \tilde b_{m,\ell} = 0,
\end{equation}
where 
\begin{equation*}
\tilde c^{m,\ell}_{q,\kappa} = r^{|q|+ 2 \kappa + 1 - (|m| + 2\ell)} c^{m,\ell}_{q, \kappa}.
\end{equation*}
The justification for \eqref{f12} is as follows. If we write  $P(x,t) + r^{k+\alpha} \tilde Q(\frac{x}{r}, \frac{t}{r^2}) = \sum_{0\le |q|+2\kappa\le k} \tilde a_{q,\kappa} x^q t^\kappa$, then it is clear that 
\begin{equation}\label{tildea}
\tilde a_{q,\kappa} = a_{q,\kappa} + r^{k+\alpha - |q| - 2\kappa} \tilde b_{q,\kappa}
\end{equation}
Now, imposing that $P(x,t) + r^{k+\alpha} \tilde Q(\frac{x}{r}, \frac{t}{r^2})$ be approximating of order $k$ for  $\frac{v}{u}$ at $(0,0)$ is, in view of \eqref{f5} and Definition \ref{D:R}, equivalent to
\begin{equation}\label{f55}
(q_n +1) (q_n +2) \tilde a_{q+ \overline{n}, \kappa} + \sum_{i \neq n} (q_i + 1) (q_i + 2) \tilde a_{q+ 2 \overline{i} - \overline{n}, \kappa}   - (\kappa+1) \tilde a_{q- \overline{n}, \kappa+1} + c^{m,\ell} _{q,\kappa} \tilde a_{m,\ell}= 0.
\end{equation}
Replacing \eqref{tildea} in \eqref{f55}, and using the fact that $P$ is itself approximating for $\frac{v}{u}$ (see \eqref{f5}), after some elementary computations we recognize that the coefficients $\tilde b_{q,\kappa}$  must  satisfy
\begin{align*}
&\ r^{k+\alpha}\bigg[r^{- (|q| + 2\kappa + 1)}(q_n +1) (q_n +2) \tilde b_{q+ \overline{n}, \kappa}
 + r^{- (|q| + 2\kappa + 1)}\sum_{i \neq n} (q_i + 1) (q_i + 2) \tilde b_{q+ 2 \overline{i} - n,\kappa} \\
 - &\  r^{- (|q| + 2\kappa + 1)}(\kappa+1) \tilde b_{q- \overline{n}, \kappa+1} + r^{-(|m|+ 2\ell)} c^{m,\ell} _{q,\kappa} \tilde b_{m,\ell}\bigg] = 0.
\end{align*}

Eliminating $r^{k+\alpha}$ and then multiplying the resulting equation by
$r^{|q| + 2\kappa +1}$ we finally obtain  \eqref{f12}. Recalling that $0<r<1$, and that from the defintion of $d_{q, \kappa}$  in \eqref{R} one has $c^{m,\ell}_{q,\kappa} \neq 0$ only when $ |m| + 2\ell \leq |q| + 2\kappa\le k-1$,  we conclude that 
\[
|\tilde c^{m,\ell}_{q,\kappa}| = r^{|q|+ 2 \kappa + 1 - (|m| + 2\ell)} |c^{m,\ell}_{q, \kappa}|  \le r |c^{m,\ell}_{q, \kappa}| \leq C \delta.
\]
Since \eqref{z10bis} above implies that $Q_0$ is approximating for $\frac{v_0}{x_n}$, then its coefficients $b_{m,\ell}$ must satisfy
\begin{equation}\label{f13}
(q_n +1) (q_n +2)  b_{q+ \overline{n},\kappa}  + \sum_{i \neq n} (q_i + 1) (q_i + 2)  b_{q+ 2 \overline{i} - \overline n,\kappa} - (\kappa+1) b_{q- \overline{n}, \kappa+1} = 0.
\end{equation}
If we now subtract \eqref{f13} from \eqref{f12}, we find that the coefficients of $\tilde Q - Q_0$ solve a linear system with left-hand side bounded by $C\delta$ and contains  unknown  coefficients $\tilde b_{m,\ell}$  such that $|m| + 2\ell$ is less than the sum of the indices  of  the  coefficients  of $\tilde Q - Q_0$  appearing in the right-hand side. We in fact have
\begin{align}
-\tilde c^{m,\ell} _{q,\kappa} \tilde b_{m,\ell} & =   (q_n +1) (q_n +2)( \tilde b_{q+ \overline{n},\kappa}-  b_{q+ \overline{n},\kappa}) 
\\
& + \sum_{i \neq n} (q_i + 1) (q_i + 2)(\tilde b_ {q+ 2 \overline{i} - \overline n,\kappa}-  b_{q+ 2 \overline{i} - \overline n,\kappa})   - (\kappa+1)( \tilde b_{q- \overline{n}, \kappa+1} - b_{q- \overline{n}, \kappa+1}).
\notag
\end{align}

The reader should note that the order of the coefficient of any term in the right hand side of the latter equation is $|q|+ 2 \kappa +1 > |m| + 2 \ell$. Again, this is so since $c^{m,\ell}_{q,\kappa} \neq 0$ precisely when $ |m| + 2\ell \leq |q| + 2\kappa$. Consequently, if  we  set  $\tilde b_{m,\ell}=b_{m, \ell}$ when $m_n=0$, then  in view of the procedure  described  after \eqref{f5}, we  can  determine all the other coefficients  $\tilde b_{q, \kappa}$ of $\tilde Q$ by induction on the order of the coefficient $|q| + 2 \kappa$.  Morever, since the coefficients of $\tilde Q-Q_0$ solve a  linear system with  left-hand side  bounded by $C\delta$, we can  further ensure that 
\begin{equation*}
||\tilde Q -  Q_0||_{L^{\infty}(Q_{1}(0,0))} \leq C \delta.
\end{equation*}
Since
\[
||\tilde Q||_{L^{\infty}(\tilde G \cap Q_{\rho}(0,0))} \le ||\tilde Q - Q_0||_{L^{\infty}(\tilde G \cap Q_{\rho}(0,0))} + ||Q||_{L^{\infty}(\tilde G \cap Q_{\rho}(0,0))},
\]
we conclude that there exists a universal constant $C>0$ such that
\[
||\tilde Q||_{L^{\infty}(\tilde G \cap Q_{\rho}(0,0))} \le C.
\]
Therefore, by \eqref{f10} and by choosing a smaller  $\delta$  if needed, one can ensure that
\begin{align}\label{f100}
||\tilde v- x_n \tilde Q||_{L^{\infty}(\tilde G \cap Q_{\rho}(0,0))} & \le  \ ||\tilde v- x_n Q_0||_{L^{\infty}(\tilde G \cap Q_{\rho}(0,0))}  + ||x_n (\tilde Q - Q_0)||_{L^{\infty}(\tilde G \cap Q_{\rho}(0,0))} 
\\
& \leq\ \frac{1}{2} \rho^{k+1 + \alpha} + C \delta \rho \le \frac{3}{4} \rho^{k+1 + \alpha}.
\notag
\end{align}
Then, again  by the second inequality  in \eqref{h1} and by \eqref{f100}, we obtain for a  smaller choice of $\delta$ if needed that
\begin{align}\label{f101}
||\tilde v- \tilde u \tilde Q||_{L^{\infty}(\tilde G \cap Q_{\rho}(0,0))} \le\ & ||\tilde v- x_n \tilde Q||_{L^{\infty}(\tilde G \cap Q_{\rho}(0,0))} + ||\tilde Q (\tilde u- x_n)||_{L^{\infty}(\tilde G \cap Q_{\rho}(0,0))}
\\
  \leq & \ \frac{3}{4} \rho^{k+1 + \alpha} + C \delta ||\tilde Q||_{L^{\infty}(\tilde G \cap Q_{\rho}(0,0))} \le \frac{3}{4} \rho^{k+1 + \alpha} + C \delta \le \rho^{k+1 + \alpha}.
\notag
\end{align}

Recalling the definitions \eqref{e21}, \eqref{tildeub} and \eqref{tildepbis}, the conclusion \eqref{vutilde2} of the lemma is now obtained as follows 
\begin{align*}
||v - u\tilde P||_{L^{\infty}(G \cap Q_{\rho r}(0,0))} = & \ ||uP + r^{k+1+\alpha} \tilde v(\frac{.}{r},\frac{.}{r^2}) - u(P + r^{k+\alpha} u \tilde Q(\frac{.}{r},\frac{.}{r^2})||_{L^{\infty}(G \cap Q_{\rho r}(0,0))}
\\
= &\ r^{k +1+ \alpha} ||\tilde v(\frac{.}{r},\frac{.}{r^2}) - \tilde u(\frac{.}{r},\frac{.}{r^2})\tilde Q(\frac{.}{r},\frac{.}{r^2})||_{L^{\infty}(G \cap Q_{\rho r}(0,0))}
\\
= &\ r^{k +1+ \alpha} ||\tilde v - \tilde u \tilde Q||_{L^{\infty}(G \cap Q_{\rho}(0,0))}\ \le (\rho r)^{k +1+ \alpha}, 
\end{align*}
where in the last inequality we have used \eqref{f101}. This completes the proof of the lemma.

\end{proof}

\medskip

\begin{proof}[Proof of  Theorem \ref{t10}]
The proof of the theorem now follows by iterating Lemma \ref{imp2}. To start  the process of  iteration,  we  take  $P= 0$. Multiplying $ v$ by a suitable constant, one can ensure  that  the hypothesis of Lemma  \ref{imp2} holds  for some $r_{0}$ universal. The rest of the proof remains the same as in the case $k=1$, in which Lemma \ref{imp2} is applied iteratively first for $r_0$, then for $\rho r_0, \rho^2 r_0$, and so on. We finally obtain a limiting polynomial $P_0$ of degree at most $k$ having the following representation
\[
P_0(x,t)= \sum_{i=1}^{\infty } (\rho^{i-1} r_0)^{k+\alpha} \tilde Q_{i}(\frac{x}{\rho^{i-1}r_0}, \frac{t}{(\rho^{i-1}r_0)^2}),
\]
where $\tilde Q_i$ is the  polynomial obtained  after the $i$-th application of Lemma \ref{imp2}. Furthermore, the  following holds  
\begin{equation}\label{interm1}
||v-uP_0||_{L^{\infty}(G \cap Q_{r}(0,0))} \leq Cr^{k+1+ \alpha},\ \ \ \ \ \ \ r\leq r_0.
\end{equation}
For $r>0$ we now consider the functions $\tilde v,  \tilde u$ defined, with $P=P_0$, as in the proof of Lemma \ref{imp2}. Given that \eqref{interm1} holds for any given $r\leq r_0$, we  can argue as in the case $k=1$ and obtain that  
\begin{equation}\label{tildes1}
|\tilde v| \leq C\tilde u\ \ \ \ \ \text{in}\  \tilde G \cap  Q_{1/2}(0,0).
\end{equation}
By rewriting $\tilde u, \tilde  v$ in terms of $u$ and $v$, from  \eqref{tildes1} we  obtain for $(x,t) \in G \cap Q_{r/2}(0, 0)$
\begin{equation}\label{tildes2}
|v(x,t) - u(x,t) P_0(x,t)| \leq C u(x,t) r^{k+\alpha}.
\end{equation}
The inequality \eqref{tildes2} gives for $(x,t) \in G \cap Q_{1}(0,0)$
\[
|\frac{v}{u}(x,t) - P_0(x,t)| \leq C (|x|^2+ |t|)^{\frac{k+\alpha}{2}}.
\]
This implies $H^{k+\alpha}$-regularity at the boundary point $(0,0)$. Finally, by arguing as in the case $k=1$, the identity
\begin{equation}\label{z25}
\frac{v(x,t)}{u(x,t)}= P_0(x,t) + r^{k+\alpha} \frac{\tilde v(\frac{x}{r},\frac{t}{r^2})}{\tilde u(\frac{x}{r},\frac{t}{r^2})},
\end{equation}
implies  that the $H^{k+\alpha}$ norm of $\frac{v}{u}$ is bounded in regions of the form $G \cap Q_{r/4}(\frac{r}{2}e_n,0)$, i.e., in regions which are  away from the boundary of $G$ by a  parabolic distance proportional to $r$. At this point, the conclusion follows similarly to the case $k=1$ by arguing as in the proof of Proposition 4.13 in \cite{CC}.

\end{proof}

\section{variable coefficients}\label{S:vc}

In this section we intend to extend Theorem \ref{main1} and Theorem \ref{t10} to variable coefficient parabolic operators of the type
\begin{equation}\label{g1}
Lu = \sum_{i,j=1}^n a_{ij}(x,t) D_{ij} u + \sum_{i=1}^n b_i(x,t) D_i u +c(x,t)u - u_t  = 0,
\end{equation}
satisfying appropriate regularity assumptions on the coefficients, and to strong solutions $u\in W^{2,1}_{p,loc}(G)$. Henceforth, we use the notation Tr$(M)$ for the trace of a matrix $M$. We can thus write $\sum_{i,j=1}^n a_{ij}(x,t) D_{ij} u = \operatorname{Tr}(A(x,t)D^2u)$, where we have indicated with $D^2 u = [D_{ij}u]$ the Hessian matrix of $u$.

As in the case of the heat equation, in order to better present the ideas we first treat the case $k=1$ in Section \ref{SS:k=1}. Subsequently, in Section \ref{SS:kbigger2}, we treat the case of $k>1$. In the sequel we will denote with $W^{2,1}_{p}(G)$ the parabolic Sobolev spaces. For their precise  definition we refer the reader to p.155 in \cite{Li}. We will indicate with $W^{2,1}_{p,loc}(G)$ the standard local spaces.

\subsection{\textbf{$H^{1+\alpha}$ regularity}}\label{SS:k=1}

The assumptions on $G$ are as in the hypothesis of Theorem  \ref{main1}. In this section we assume that $A = [a_{ij}]\in H^{ \alpha}(G)$, $b, c \in L^{\infty}(G)$. The following is our main result.

\begin{thrm}\label{t1}
Let $p>1$ and suppose that $u\in W^{2,1}_{p, loc}(G \cap Q_{2}(0,0))$ be a positive strong solution to \eqref{g1} above.
Let $v\in W^{2,1}_{p}(G\cap Q_{2}(0,0))$ be a strong solution to 
\begin{equation}
Lv=g,\  \text{in}\ G, 
\end{equation}
where $g\in H^{\alpha}(\overline{G\cap Q_{2}(0,0)})$. Assume that  $u, v$  vanish on $\partial_{p} G \cap Q_{2}(0,0)$.  Furthermore, let  $u$  satisfy the normalization condition \eqref{assump}. Then, one has
\begin{equation}
||\frac{v}{u}||_{H^{1+\alpha}(G \cap Q_{1}(0,0))} \leq C (||v||_{L^{\infty}(G\cap Q_{2}(0,0))}  + ||g||_{H^\alpha(G\cap Q_{2}(0,0))} + 1),
\end{equation}
for some $C>0$ universal.
\end{thrm}

\begin{rmrk}
We first note that  from the assumptions on the coefficients, the Calder\'on- Zygmund theory implies that $u \in W^{2,1}_{q, loc}(G)$ for all $1<q<\infty$, see for instance Proposition 7.14 in \cite{Li}. Then, we can invoke Theorem 4.29 in \cite{Li} to  conclude that   $v, u$ are in $H^{1+\alpha}(\overline{G\cap Q_{2}(0,0)})$. We note that Theorem 4.29 in \cite{Li} can be applied to strong  solutions  in $W^{2,1}_{n+1, loc}$ via approximations  by solutions to equations with smooth coefficients and by an application of the comparison principle Theorem 7.1 in \cite{Li}. We refer to \cite{GH} for the elliptic  counterpart of  such intermediate  Schauder  type regularity result.
\end{rmrk}
After a  suitable change of coordinates and parabolic dilation similar to \eqref{d1}, we may assume that
\begin{equation}\label{g2}
\begin{cases}
 A(0, 0) = I,
 \\ f(0)=0,\ \ \ D' f(0)=0,\ \ \ ||f||_{C^{1,\alpha}(G \cap Q_{2}(0,0))}\le \delta,
\\   
Du(0, 0)=e_n,\ \ \ \ ||u-x_n||_{H^{1+\alpha}(G \cap Q_{2}(0,0))} \le \delta,
\\
\max \bigg\{ [A]_{\alpha, G \cap Q_{2}(0,0)}, ||g||_{H^{\alpha}(G \cap Q_{2}(0,0))}, ||b, c||_{L^{\infty}(G \cap Q_{2}(0,0))}\bigg\} \leq \delta.
\end{cases}
\end{equation}
Here, $[A]_{\alpha}$ indicates the $\alpha$-H\"{o}lder seminorm of $[a_{ij}]$, see p. 46 in \cite{Li}. More precisely, first by a suitable change  of coordinates, we can  ensure that $A(0, 0)=I$. Then, by letting
\begin{equation}\label{d12}
 u_{r_0}(x, t)= \frac{ u(r_{0} x, r_{0}^{2} t)}{r_{0}}\ \ \ \ \ v_{r_0}(x, t)= \frac{ v(r_{0} x, r_{0}^{2} t)}{r_{0}},
\end{equation}
as in \eqref{d1}, we have that  $u_{r_0}, v_{r_0}$  solve in $G^{r_0}$ (same definition as in Section \ref{r1})
\[
L_0 u_{r_0} = 0,\ \ \ \ L_0 v_{r_0}= g_0,
\]
where
\[
L_0 w= \operatorname{Tr}( A_0 D^2 w) +  <b_0,D w> + c_0 w -  w_t,
\]
and
\[
A_{0}(x, t)= A(r_0 x, r_{0}^{2} t),\  b_{0}(x, t)= r_0 b(r_{0}x, r_{0}^{2}t),\ c_0=r_{0}^2 c(r_{0}x, r_{0}^{2}t),\ g_0(x,t)=r_0 g(r_{0}x, r_{0}^{2}t).
\]
Therefore, if $r_0$ is suitably chosen depending on $\delta$,  \eqref{g2} can be ensured. As before, by abuse of notation, we keep calling $u_{r_0}=u$, $G^{r_0}=G$ and so on. Moreover, as in Section \ref{r1}, $\delta$ will be determined later.

We  now  introduce the  relevant notion of approximating function with respect to  $L$ and $g$, where $g$ is as in Theorem \ref{t1}.

\begin{dfn}
 We say that $P(x) = a_0 + \sum_{i=1}^n a_i x_i$ is an approximating  affine function at $(0,0)$ for $\frac{v}{u}$ with respect to $L, g$ if  $2 a_n = g(0,0)$. 
\end{dfn}

With this  notion the corresponding statement  of Lemma \ref{imp1} remains the same,  but  its proof needs to be slightly modified.

\begin{lemma}\label{imp5}
Let $u,v$ be as in Theorem \ref{t1}. Assume that for  some $r \leq 1 $ and $P(x) = a_0 + \sum_{i=1}^n a_i x_i$ an  approximating affine function at $(0,0)$ for $\frac vu$ (with respect to $L$ and $g$) with $|a_i|\leq 1$, one has
\begin{equation}\label{i1b}
||v - uP||_{L^{\infty}(G \cap Q_{r}(0,0))} \leq r^{2+ \alpha}.
\end{equation}
Then, there exists an  approximating affine function $\tilde P$ such that for some $C, \rho > 0$ universal, we have  
\begin{equation}\label{pptildebis}
 ||P-\tilde P||_{L^{\infty}(G \cap Q_{r}(0,0))} \leq C r^{1+\alpha},
\end{equation}
and
\begin{equation}\label{vutildebis}
||v - u\tilde P||_{L^{\infty}(G \cap Q_{\rho r}(0,0))} \leq (\rho r)^{2+ \alpha}.
\end{equation}
\end{lemma}

\begin{proof}

We point out the essential modifications in the proof of  Lemma \ref{imp1}  in the present context.
 Let $\tilde v$ and $\tilde u$ be as in the proof of Lemma \ref{imp1}. Then, one has
\begin{equation}\label{g3}
g= Lv= L(uP) + r^{\alpha} \tilde L \tilde v (x/r, t/r^2),
\end{equation}
where 
\begin{equation}\label{g10}
\tilde L \tilde v = Tr(\tilde A D^2 \tilde v) + r <\tilde b,D \tilde v> + r^2 \tilde c \tilde v - \tilde v_t,
\end{equation}
and
\begin{equation}\label{g11}
\tilde A= A(rx, r^2 t),\ \tilde b= b(rx, r^2 t),\  \tilde c= c(rx, r^2 t),\  \ \ (x,t) \in \tilde G.
\end{equation}

Since $Lu=0$, one has
\[
L(uP)= 2 <A Du,DP>  + u <b,DP>.
\]

From \eqref{g2} we have for all $(x,t) \in G \cap Q_{r}(0,0)$
\begin{equation}\label{g50}
\begin{cases}
A(x,t)= I + \delta r^{\alpha} M(x,t),
\\   
Du(x,t)= e_n +\delta  r^{\alpha} w(x, t), 
\\  g(x,t)= g(0,0)  + \delta r^{\alpha} w_1(x, t), 
\\
|u(x,t)| \leq C  \delta r,
\end{cases}
\end{equation}
where $M,w, w_1$ are bounded.
Therefore, using \eqref{g50} and \eqref{g2} and the fact that $P$ is approximating for $\frac{v}{u}$, we obtain for some bounded function $K(x,t)$,
\[
|L(uP)-g|= | 2<ADu, DP> + u<b,DP> - g|=|g(0,0) - g(x,t) + \delta r^{\alpha} K(x,t)| \leq  K_1 \delta r^{\alpha}.
\]
Combined with \eqref{g3} this estimate gives
\[
|\tilde L  \tilde v| \leq C \delta\ \ \ \text{in}\  \ \ \tilde G \cap Q_{1}(0,0).
\]
We also note that,  with $\tilde u$ as in \eqref{e2}, we have
\[
\tilde L \tilde u=0.
\]
Letting  $\delta \to 0$, from \eqref{g2} we obtain that  up to a subsequence $\delta \to 0$, $\tilde v = \tilde v(\delta) \to  v_{0}$, which solves \eqref{s}. Note that, unlike the case of the heat equation, we do not presently have uniform interior $H^{2+ \alpha}$ estimates for $\tilde v$. Nevertheless, because of $H^{1+\alpha}$ estimates for $\tilde  v= \tilde v(\delta)$ up to $\partial_{p} \tilde G \cap Q_{1}(0,0)$ independent of $\delta$, by applying Theorem 6.1 in \cite{CKS} we can ensure that $v_{0}$ is a $L^{p}$ viscosity solution of the heat equation in the sense of \cite{CKS}. The regularity theory  for viscosity solutions now ensures that $v_{0}$ is a classical solution of the heat equation. As in the proof of Lemma \ref{imp1} we have that  \eqref{e4}-\eqref{e7} holds for $Q_0(x)= \sum_{i=1}^n q_i x_i + q_0$, with $q_n=0$. Then, as in the case of heat equation the conclusion of the lemma follows  with $\tilde P= P + r^{1+\alpha} Q_{0}(\frac{.}{r})$. Note that, if we let $P(x)= \sum_{i=1}^n a_i x_i + a_0$ and $\tilde P(x)= \sum_{i=1}^n \tilde a_i x_i + \tilde a_0$, then since $P$ is an approximating affine function for $L$ and $g$ we have $a_n= \frac{g(0,0)}{2}$. Since $q_n=0$, we have
\[
\tilde a_n= a_n + r^{\alpha}q_n= \frac{g(0,0)}{2}.
\]
This shows that also $\tilde P$ is an approximating affine function at $(0,0)$ for $\frac vu$ with respect to $L$ and $g$. From this fact, the verification of \eqref{vutildebis} above follows as that of \eqref{vutilde} in Lemma \ref{imp1}.

\end{proof}

\begin{proof}[Proof of Theorem \ref{t1}] We repeatedly apply Lemma \ref{imp5}. To start the process we  first take  $P(x)= \frac{g(0,0)}{2}x_n$. Then, \eqref{i1b} holds for some  universal $r=r_{0}$ when $v, g$ and $P$  are  multiplied by suitable constants. As  before, by iterating  Lemma \ref{imp5} with $ r=r_0, \rho r_0, \rho^2 r_0$ and so on, we  obtain a  limiting affine function $P_0$ such that \eqref{tilde1} holds. We note that in this case, $P_0$ has the following explicit  representation
\begin{equation}
P_0(x)= P(x) + \sum_{i=1}^{\infty} (\rho^{i-1}r_0)^{1+\alpha} Q_{i}(\frac{x}{\rho^{i-1}r_0})
\end{equation}
where $P(x)= \frac{g(0,0)}{2}x_n$, and $Q_i(x)$ is the affine  function obtained in the $i$-th iteration of Lemma \ref{imp5}. The rest of  the proof remains the   same as that for the heat equation.

\end{proof}

\subsection{\textbf{$H^{k+\alpha}$ regularity for $k \geq 2$}}\label{SS:kbigger2}

In what follows the assumptions on $G$ are as in Section \ref{s2}. We  have the following  higher-regularity result for variable coefficient operators. We note that in the next result we do not assume that $u$ and  $v$  are  strong solutions as in Theorem 4.1 since by the regularity  theory one infers that both $u$ and $v$ are classical solutions.

\begin{thrm}\label{t4}
Let $u$ and $v$ be (classical) solutions in $G \cap Q_{2}(0,0)$ of the equations
\begin{equation}\label{g100}
Lu = \operatorname{Tr}(A D^2 u) + <b,Du> + cu - u_t = 0,
\end{equation}
and
\begin{equation}
Lv=g,
\end{equation} 
where $A, g \in H^{ k-1 + \alpha}(\overline{G \cap Q_{2}(0,0)})$, $ b, c \in H^{k-2+ \alpha}(\overline{G \cap Q_{2}(0,0)})$. Assume that  $u, v$  vanish  on $\partial_{p} G \cap Q_{2}(0,0)$. Also, let $u>0$ in $G \cap Q_{2}(0,0)$, and assume furthermore  that it satisfy the normalization condition \eqref{assump}. Then, one has 
\begin{equation}
||\frac{v}{u}||_{H^{k+\alpha}(G \cap Q_{1}(0, 0))} \leq C (||v||_{L^{\infty}(G \cap Q_{2}(0,0))} + ||g||_{H^{k-1+ \alpha}(G \cap Q_{2}(0,0))} + 1),
\end{equation}
for some $C>0$ universal.
\end{thrm}

As  before,  from the Schauder theory as in Chapters 4 and 5 in  \cite{Li}, we have that $u, v \in H^{k+ \alpha}(\overline{G \cap Q_{3/2}(0,0)})$. By a suitable change of  coordinates and parabolic dilations similar to \eqref{d1}, we can assume that 
\begin{equation}\label{h200}
\begin{cases}
A(0, 0)= I,\ \ \ \   ||A-I||_{H^{k-1+\alpha}(G \cap Q_{2}(0,0))} \leq \delta,
\\
f(0)=0,\ \ \ D' f(0)=0,\ \ \  ||f||_{H^{k +\alpha}(G \cap Q_{2}(0,0))}\le \delta,
\\
Du(0, 0)=e_n, \ \  \ ||u-x_n||_{H^{k+\alpha}(G \cap Q_{2}(0,0))}\leq  \delta,
\\
\max \bigg\{|| g||_{H^{k-1+\alpha}(G \cap Q_{2}(0,0))}, ||b, c||_{H^{k-2+\alpha}((G \cap Q_{2}(0,0))} \bigg\} \leq  \delta,
\end{cases}
\end{equation}
where $0<\delta<1$ is to be chosen appropriately later.

Similarly to what was done in the proof of Theorem \ref{t10} above, we need to compute $L(uP)$, where $P(x,t) = \sum_{0\le |m|+2\ell\le k} a_{m,\ell}x^m t^{\ell}$ is a parabolic polynomial of degree $k$. Since by \eqref{g100} we have $Lu=0$, we obtain 
\begin{equation}\label{h201}
L(uP)=  u LP + 2 <A Du,DP>  + <b,DP>.
\end{equation}
By the linearity of $L$ we are thus led to understand \eqref{h201} when $P(x,t) = x^{m} t^{\ell}$. In such case, from \eqref{h200} it follows that 
\begin{align}\label{h102}
L(uP) = &\  m_n (m_{n}+1)x^{m-\overline{n}} t^{\ell} + \sum_{i \neq n} m_i (m_i - 1) x^{m - 2 \overline{i} + \overline{n}} t^{\ell}  - \ell x^{m+ \overline{n}} t^{\ell-1}
\\
 + & \sum_{|m| + 2 \ell \leq |q|+2\kappa\le k-1} c^{m , \ell}_{q, \kappa} x^q t^\kappa+  w_{m, \ell},
\notag
\end{align}
where because of  the representation \eqref{h100} which is valid for $u$,   \eqref{h200} and  \eqref{h201}, we have that
\begin{equation}\label{h5}
| c^{m, \ell} _{q, \kappa}| \leq C\delta, |w_{m , \ell}| \leq C \delta ( (|x|^2 + |t|)^{\frac{k-1+ \alpha}{2}}), \ \text{and}\ ||w_{m,\ell}||_{H^{k-2 + \alpha} (G \cap Q_r)}  \leq C \delta r.
\end{equation}
In \eqref{h102}, we used the fact that although the term $u<b,DP> \in H^{k-2+\alpha}(G \cap Q_{2}(0,0))$, but nevertheless  \eqref{h102} and \eqref{h5} can be justified as follows.  We first  note (see also \eqref{uxn}
 above), that because  of \eqref{h200}, we have
\begin{equation}\label{int1}
 u(x,t)=x_n +  P_1(x,t) + w_1(x,t),
\end{equation}
 where $P_1$ is a polynomial such that  $2 \leq \operatorname{deg}(P_1) \leq k$  and 
\begin{equation*}
||P_1|| \leq C \delta,\ \ \ \ \ |w_{1}(x,t)| \leq C \delta (|x|^2 + |t|)^{\frac{k+\alpha}{2}}.
\end{equation*}
 We furthermore  note that, since $b \in H^{k-2+\alpha}(\overline{G \cap Q_2(0,0)})$, we can write
\begin{equation}\label{int2}
b(x,t)= P_b(x,t) +  b_1(x,t),
\end{equation}
 where  $P_b$ is a  vector field  in $\Rn$ each of whose  components are  polynomials of degree at most $k-2$. Moreover, because of \eqref{h200} the following holds 
\begin{equation*}
||P_b|| \leq C \delta,\ \ \ \ \ |b_{1}(x,t)| \leq C \delta (|x|^2 + |t|)^{\frac{k -2 +\alpha}{2}}.
\end{equation*}
Therefore, from \eqref{int1} and \eqref{int2} it follows that 
\begin{equation*}
u<b, DP>= P_{u,b}(x,t) + w_{u,b}(x,t),
\end{equation*}
where $P_{u,b}(x,t)$ is a polynomial of degree  at most $k-1$ such that
\begin{equation*}
||P_{u,b}|| \leq C \delta,\ |w_{u,b}(x,t)| \leq C \delta (|x|^2 + |t|)^{\frac{k -1+\alpha}{2}}.
\end{equation*}

We also observe that in \eqref{h5} we have that  $c^{m , \ell}_{q, \kappa} \neq 0$ only when $|m| + 2 \ell \leq |q| + 2 \kappa \leq k-1$, and that, furthermore, the coefficient does depend on $A, u$ and $b$. 
Now  for  a  general polynomial of the  form $\sum_{0\le |m|+2\ell\le k} a_{m,\ell} x^m t^{\ell}$ one has
\begin{equation}\label{h2}
L(uP)=  R(x,t) + \sum a_{m,\ell} w_{m,\ell}(x,t),
\end{equation}
where
\begin{equation}\label{RR}
R(x,t)=\sum_{0\le |q|+2\kappa \le k-1} d_{q, \kappa} x^{q} t^{\kappa},
\end{equation}
and  the coefficients $a_{m , \ell}$ and $d_{q, \kappa}$ of $P$, and $c^{m,\ell}_{q, \kappa}$  satisfy \eqref{f5} as in the case of heat equation.

W e next introduce the appropriate notion of approximating polynomial in the present context.

\begin{dfn}
We  say that a parabolic polynomial $P$ of degree $\leq k$ is approximating of order $k$ at $(0,0)$ for $\frac{v}{u}$  with respect to $L$ and $g$ if the coefficients $d_{\ell, m}$ of $R(x,t)$ in the representation \eqref{h2}, \eqref{RR} above coincide with the coefficients of the Taylor polynomial of order  $k-1$ for $g$ at $(0,0)$. 
\end{dfn}

 With this notion in place,  we now  state the  analogue of Lemma \ref{imp2}.

\begin{lemma}\label{imp6}
Let $u,v$ be as in Theorem \ref{t4}. Assume that for  some $r \leq 1 $ and $P$ an  approximating polynomial of order $k$ for $\frac{v}{u}$ at $(0,0)$ with respect to $L$ and $g$,  with $||P||\leq 1$, one has
\begin{equation}\label{i1bis}
||v - uP||_{L^{\infty}(G \cap Q_{r}(0,0))} \leq r^{k+1+ \alpha}.
\end{equation}
Then, there exists an  approximating polynomial  $\tilde P$ of order $k$ such that for some $C, \rho > 0$ universal, we have  
\begin{equation}\label{pptildebis}
 ||P-\tilde P||_{L^{\infty}(G \cap Q_{r}(0,0))} \leq C r^{k+\alpha},
\end{equation}
and
\begin{equation}\label{vutildebis}
||v - u\tilde P||_{L^{\infty}(G \cap Q_{\rho r}(0,0))} \leq (\rho r)^{k +1+ \alpha}.
\end{equation}
\end{lemma}
\begin{proof}
 The proof of Lemma \ref{imp6} follows  by arguing as in that of Lemma \ref{imp2}. We define
\begin{equation*}
v= uP  + r^{k+1+\alpha} \tilde v(\frac{.}{r},\frac{.}{r^2}),
\end{equation*}
where $P$ satisfies the hypothesis of the lemma. Since $g \in H^{ k-1 + \alpha}(\overline{G \cap Q_{2}(0,0)})$,  from the bounds in \eqref{h5} we have for all $(x,t) \in G \cap Q_{r}(0,0)$,
\begin{equation}\label{h101}
|g(x,t) - P_{g}(x,t)| \leq C \delta (|x|^2+ |t|)^{\frac{k-1+\alpha}{2}},
\end{equation}
where $P_g$ is a  polynomial of degree at most $k-1$.

 We note that  since   $P$ is approximating for $\frac{v}{u}$ at $(0,0)$ with respect to $L$ and $g$, we  have that $P_{g}(x,t)= R(x,t)$. Therefore, from \eqref{h102} and the bounds in \eqref{h5} we obtain in $G \cap  Q_{r}(0,0)$
\begin{equation}\label{bd}
|r^{k-1+\alpha} \tilde L  \tilde v(\frac{x}{r}, \frac{t}{r^2})| =|L(uP)- g| \leq  C \delta  r^{k-1+\alpha},
\end{equation}
where $\tilde L$ is as in \eqref{g10}. The estimate \eqref{bd} implies 
\begin{equation}\label{bd1}
|\tilde  L \tilde v| \leq C \delta,
\end{equation}
where  $\tilde L$  is as in  \eqref{g10}. As a consequence, for a subsequence $\delta \to 0$ we have $\tilde v = \tilde v(\delta) \to v_0$, where  $v_0$ is as in  \eqref{s1}. Now, similarly to the proof of Lemma \ref{imp2} there exists $Q_0$ such that \eqref{z10bis}-\eqref{f11} holds. Moreover, as in the case of heat equation, the polynomial $P(x,t) + r^{k+\alpha}Q_0(\frac{x}{r}, \frac{t}{r^2})$ need not be  approximating for $\frac{v}{u}$  with respect to $L$ and $g$. Therefore as before, we modify $Q_0$ to $\tilde Q$ such that $P(x,t) + r^{k+\alpha}\tilde Q(\frac{x}{r}, \frac{t}{r^2})$  is an approximating polynomial of order $k$ for $\frac{v}{u}$ at $(0,0)$ with respect to $L$ and $g$. Since $P$ is already  an approximating polynomial for $\frac{v}{u}$ the coefficients of $\tilde Q$ should satisfy \eqref{f12} similarly to the situation of the heat equation. The only difference in the present case being that in the analogue of \eqref{f12} the coefficients $\tilde c^{m, \ell}_{q, \kappa}$ would additionally depend on $A$ and $b$, besides $u$. The rest of the arguments remain the same as  in the proof of Lemma \ref{imp2},  and  the desired conclusion  follows.

\end{proof}

\begin{proof}[Proof of Theorem \ref{t4}] As previously, it follows by applying  Lemma \ref{imp6} repeatedly. In this case,  in order to start the process of iteration  we determine an approximating polynomial  $P= \sum_{m,\ell} a_{m,\ell} x^m t^\ell$ from  \eqref{f5} where $d_{q,\kappa}$'s  are  determined by the Taylor  polynomial $P_{g}(x,t)$ of order $k-1$ for $g$. In view of the procedure described after \eqref{f5}, such a  polynomial can be determined.  Then,  by  multiplying $v$ and $P$ by a  suitable constant,  the hypothesis of Lemma \ref{imp6}  holds for some small enough universal $r_{0}$. Therefore,  by applying the Lemma \ref{imp6} iteratively with $r=r_0, \rho r_0, \rho^2 r_0$ and so on, we  obtain a limiting polynomial $P_0$ which has the following representation
\begin{equation}
P_0(x,t) = P(x,t) + \sum_{i=1}^{\infty } (\rho^{i-1} r_0)^{k+\alpha} \tilde Q_{i}(\frac{x}{\rho^{i-1} r_0}, \frac{t}{(\rho^{i-1} r_0)^2}),
\end{equation}
where $P$ is the above polynomial which is determined before  the first step of the iteration and $\tilde Q_i$ are the polynomials determined after the $i$-th application of Lemma \ref{imp6}. Moreover,  with such a  $P_0$, we have  that \eqref{interm1} holds. The rest of the proof remains the same as that for the heat equation.

\end{proof}

\section{Application to the parabolic obstacle  problem}\label{S:pob}

As mentioned in the introduction, we  close the paper  with an application of Theorem \ref{t10} to the parabolic  obstacle. For a measurable set $E\subset \R^{n+1}$, we indicate with $\chi_{\Om}$ its indicator function. We consider  the following  problem studied in \cite{CPS}.

\begin{prob}\label{P:pb}
Given a domain $D\subset \Rn \times \R$, consider a function $u(x,t)$ defined in $D$ such that $u, Du$ are  continuous, and define the \emph{coincidence set} as 
\begin{equation*}
\Lambda=\{(x,t) \in D \mid u(x,t)=Du(x,t)=0\}.
\end{equation*}
With $\Om= D \setminus \Lambda$, suppose that $u$ solves the equation
\begin{equation*}
\Delta u -u_t= \chi_{\Om}.
\end{equation*}
The  \emph{free  boundary}  is defined as
\begin{equation*}
\Gamma= \Gamma(u)= \partial \Om \cap D.
\end{equation*}

\end{prob}

With this  setup, we  mention a corollary of Theorem \ref{t10}.
\begin{cor}\label{fb10}
Let $(x_0,t_0)\in \Gamma$ be a point as in Theorem 13.1 and Lemma 13.3 in \cite{CPS}. Then, $\partial \Om \cap Q_{1/4}(x_0, t_0)$ is  $C^\infty$.
\end{cor}

\begin{proof}
By the $C_{x}^{1,1}$  regularity of $u$ as  in Section 4 in \cite{CPS}, we have that  the spatial  derivatives  $D_i u$ vanish on $\partial \Om$ for all $i=1, ..n$. Moreover, Lemma 13.3 in \cite{CPS} implies that $D_t u$ also  vanishes continuously at the free  boundary  $\Gamma  \cap Q_{1/4}(x_0,t_0)$ near $(x_0,t_0)$. Now, Theorem 14.1  in \cite{CPS} implies that $Q_{1/4}(x_0,t_0) \cap  \Gamma$ is $C^{1,\alpha}$ regular which follows by an application of the boundary Harnack inequality as in \cite{ACS}.  Moreover, in the  proof  of  Theorem 14.1 in \cite{CPS} it is  evident that, without loss of generality, one can assume that  $\Gamma \cap Q_{\rho/4}(x_0,t_0)=\{(x,t) \mid x_n = f(x', t) \}$ and  that  the following holds, 
\begin{equation}\label{a10}
u_n( x_0+ \frac{3}{16}e_n, t_0 -(\rho/16)^2) \geq c_0,
\end{equation}
 for some $c_0, \rho >0$ universal. Moreover, $\rho$ can be chosen  in such a  way that the point  $(x_0+ \frac{3}{16}e_n,t_0 -(\rho/16)^2)$ is at a parabolic distance from $\Gamma$ bounded from below by a  universal constant $C_0$. Now
\begin{equation}\label{fb8}
u(x',f(x',t),t)=0.
\end{equation}
Therefore,  by differentiating the equation \eqref{fb8} with respect to the variables $x_1,...x_{n-1}, t$, we  obtain that
\begin{equation}\label{fb}
\frac{D_i u}{D_n u}= D_i f,\  \frac{D_t u}{D_n u}= D_t f.
\end{equation}
Since  \eqref{a10} is a  scaled version of  the normalization \eqref{assump}, this implies that if we take  $v= D_i u$ and $u=D_n u$ in Theorem \ref{main1}, we obtain from \eqref{fb} that $D'f \in H^{1+\alpha}$. Similarly, with $v=D_t u$ and $u=D_n u$, by application of Theorem \ref{main1} we find that $D_t f \in H^{1+\alpha}$. This implies that $f \in H^{2+\alpha}$, i.e.,  the free boundary is $H^{2+\alpha}$ regular. We now proceed  inductively as follows. Suppose we  know that $f$ and hence the free boundary is in $H^{k+\alpha}$ for some $k \geq 2$. Then, by applying Theorem \ref{t10} to $v=D_i u$ and $u=D_n u$, we obtain from \eqref{fb} that $D'f \in H^{k+\alpha}$. Similarly, with $v=D_t u$ and $u=D_n u$, we find that $D_t f \in H^{k+\alpha}$. This clearly  implies that $f \in H^{k+1+\alpha}$  and hence the free  boundary $\Gamma \cap Q_{1/4}(x_0,t_0)$ is $H^{k+1+\alpha}$. Therefore,  we can  repeatedly  apply Theorem \ref{t10} to conclude that $\Gamma \cap Q_{1/4}(x_0,t_0)$ is smooth.

\end{proof}

\begin{rmrk}
As mentioned before in the introduction, one can in fact establish space-like real analyticity of the free boundary by employing  the hodograph transform in \cite{CPS} (see Theorem 15.1 in \cite{CPS}). Nevertheless, similarly to the elliptic case, the proof of Corollary \ref{fb10} provides a new perspective in the study of parabolic free boundary  problems. It remains an interesting question to see  if one can establish an analogue of Theorem \ref{t10} and Corollary \ref{fb10} in the thin parabolic obstacle problem studied in \cite{DGPT} where the use of the hodograph transformation does not appear feasible.
\end{rmrk}

\end{document}